\documentclass{amsart}
\usepackage{amssymb}
\usepackage{hyperref}
\usepackage{url}
\usepackage{enumitem}
\usepackage{epsfig}  		% For postscript

\usepackage[style=ext-numeric,
sorting=nyt,
url=false,
doi=false,
isbn=false,
articlein=false,
date=year,
giveninits=true
]{biblatex}

\renewbibmacro*{volume+number+eid}{%
  \printfield{volume}%
  \setunit*{\adddot}% DELETED
  \setunit*{\addcomma\space}% NEW (optional); there's also \addnbthinspace
  \printfield{number}%
  \setunit{\addcomma\space}%
  \printfield{eid}
 }
 
\DeclareFieldFormat{pages}{#1}
\DeclareFieldFormat*{title}{\mkbibitalic{#1}}
\DeclareFieldFormat*{citetitle}{\mkbibitalic{#1}}
\DeclareFieldFormat{journaltitle}{#1\isdot}

\makeatletter
\newtheorem*{rep@theorem}{\rep@title}
\newcommand{\newreptheorem}[2]{%
\newenvironment{rep#1}[1]{%
 \def\rep@title{#2 \ref{##1}}%
 \begin{rep@theorem}}%
 {\end{rep@theorem}}}
\makeatother

\newtheorem{thm}{Theorem}[section]
\newreptheorem{thm}{Theorem}
\newtheorem{prop}[thm]{Proposition}
\newreptheorem{prop}{Proposition}
\newtheorem{lem}[thm]{Lemma}
\newreptheorem{lem}{Lemma}
\newtheorem{cor}[thm]{Corollary}
\newreptheorem{cor}{Corollary}
\newtheorem{conj}[thm]{Conjecture} 
\newtheorem{quest}[thm]{Question} 
\theoremstyle{definition}

\theoremstyle{remark}
\newtheorem{remark}[thm]{Remark}

\numberwithin{equation}{section}

\newcommand{\Q}{\mathbb{Q}}  % The rationals.
\newcommand{\Z}{\mathbb{Z}}  % The integers.
\newcommand{\bsig}{\mathbf{\sigma}}  % Choices of 0's and 1's.

\begin{document}

\title{Residual Torsion-Free Nilpotence, Bi-Orderability and Pretzel Knots}

\author{Jonathan Johnson}
\address{Department of Mathematics, University of Texas at Austin, 
Austin, TX}
\email{jonjohnson@utexas.edu}
%\urladdr{www.math.sc.edu/$\sim$howard} % Delete if not wanted.

\begin{abstract}
The residual torsion-free nilpotence of the commutator subgroup of a knot group has played a key role in studying the bi-orderability of knot groups \cite{PerRolf03,LRR08,CDN16,John20a}.
A technique developed by Mayland \cite{May75} provides a sufficient condition for the commutator subgroup of a knot group to be residually-torsion-free nilpotent using work of Baumslag \cite{Baum67, Baum69}.
In this paper, we apply Mayland's technique to several genus one pretzel knots
and a family of pretzel knots with arbitrarily high genus.
As a result, we obtain a large number of new examples of knots with bi-orderable knot groups.
These are the first examples of bi-orderable knot groups for knots which are not fibered or alternating.
\end{abstract}

\maketitle

%\tableofcontents

%%%%%%%%%%%%%%%%%%%%%%%%%%%%%%%%%%%%%%%%%%%%%%%%%%%%%%%%%%%%%%%%%%%%%%
\section{Introduction}
%%%%%%%%%%%%%%%%%%%%%%%%%%%%%%%%%%%%%%%%%%%%%%%%%%%%%%%%%%%%%%%%%%%%%%

Let $J$ be a knot in $S^3$.
The \emph{knot exterior} of $J$ is $M_J:=S^3-\nu(J)$ where $\nu(J)$ is the interior of a tubular neighborhood of $J$,
and the \emph{knot group} of $J$ is $\pi_1(M_J)$.
Denote the Alexander polynomial of $J$ by $\Delta_J$.

A group $\Gamma$ is \emph{nilpotent} if it's lower central series terminates (is trivial) after finitely many steps.
In other words, for some non-negative integer $n$,
\[
    \Gamma_0 \rhd \Gamma_1 \rhd \cdots \rhd \Gamma_n = 1 
\]
where $\Gamma_0=\Gamma$ and $\Gamma_{i+1}=[\Gamma_{i},\Gamma]$ for each $i=0,\ldots n-1$.
A group $\Gamma$ is \emph{residually torsion-free nilpotent} if for every nontrivial element $x\in\Gamma$, there is a normal subgroup $N\lhd\Gamma$ such that $x\notin N$ and $G/N$ is a torsion-free nilpotent group.
This paper concerned with when the commutator subgroup of a knot's group is residually torsion-free nilpotent, which has applications to ribbon concordance \cite{Gor81} and the bi-orderability of the knot's group \cite{LRR08}.

Several knots are known to have groups with residually torsion-free nilpotent commutator subgroups.
The commutator subgroup of fibered knot groups are finitely generated free groups, which are residually torsion-free nilpotent \cite{Mag35}.
Work of Mayland and Murasugi \cite{MayMur76} shows that the knot groups of pseudo-alternating knots, whose Alexander polynomials have a prime power leading coefficient,
have residually torsion-free nilpotent commutator subgroups;
pseudo-alternating knots are defined in Section \ref{secgenusone}.
The knot groups of two-bridge knots have residually torsion-free nilpotent commutator subgroups \cite{John20a}.

There is also the following obstruction to a knot's group having residually torsion-free nilpotent commutator subgroup.

\begin{prop} \label{trivial}
If $J$ is a knot in $S^3$ with trivial Alexander polynomial, then the commutator subgroup of $\pi_1(M_J)$ cannot be residually torsion-free nilpotent.
\end{prop}

\begin{proof}
Let $G$ be the commutator subgroup of $\pi_1(M_J)$.
Let $M^{\infty}$ be the infinite cyclic cover of $M_J$,
the covering space of $M_J$ corresponding to $G$ so that  $\pi_1(M^{\infty})=G$;
see \cite[Chapter 7]{Rolf76} for details.
\[
H_1(M^{\infty},\Z)\cong \bigoplus_{i=1}^n \Z[t,t^{-1}]/\langle a_i(t)\rangle
\]
where $a_1(t),\ldots,a_n(t)$ are polynomials such that
\[
\prod_{i=1}^n a_i(t)=\Delta_J(t).
\]

Since the Alexander polynomial of $J$ is trivial, $G/[G,G]\cong H_1(M^{\infty},\Z)=1$
so $G=[G,G]$.
It follows that every term of the lower central series of $G$ is isomorphic to $G$.
Suppose $N\lhd G$ is a proper normal subgroup of $G$.
For each term of the lower central series of $G/N$,
\[
    (G/N)_i\cong G_i/N\cong G/N\neq 1
\]
so $G/N$ cannot be nilpotent.
Thus, $G$ is not residually torsion-free nilpotent.
\end{proof}

\begin{figure}[t]
\includegraphics[scale=1.0]{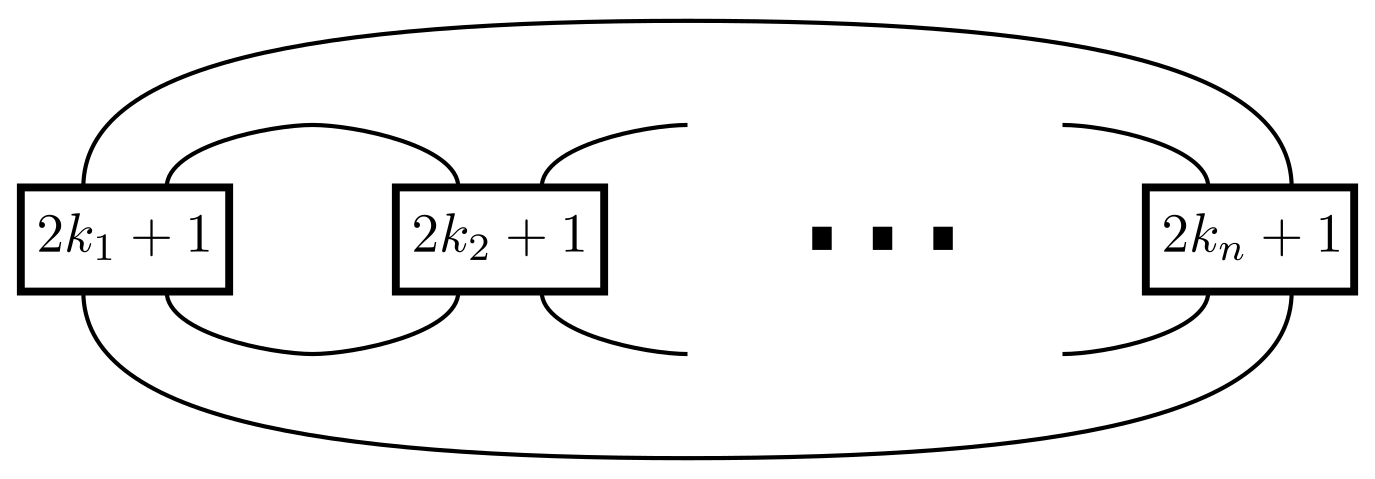}
\caption{A Pretzel Knot Diagram: The integers in the boxes indicate number of right-hand half-twist when positive and left-hand half-twist when negative.}
\label{figpretzel}
\end{figure}

Given the integers $k_1, k_2,\ldots,k_n$, define $P(k_1,k_2,\ldots,k_n)$ to be the \emph{pretzel knot} represented in the diagram in Figure \ref{figpretzel}.
Mayland \cite{May75} describes a technique to examine the commutator subgroup of the group of a knot bounding an unknotted minimal genus Seifert surface; see section \ref{secmayland}.
In fact, this is the technique Mayland and Murasugi used to prove their result for pseudo-alternating knots \cite{MayMur76}.
Applying Seifert's algorithm to the diagram in Figure \ref{figpretzel} yields an unknotted minimal genus Seifert surface \cite{Gab86} making pretzel knots ideal candidates for Mayland's technique.

Let $J$ be the $P(2p+1,2q+1,2r+1)$ pretzel knot for some integers $p$, $q$, and $r$.
$J$ is a two-bridge knot (possibly trivial) precisely when at least one of $p$, $q$, and $r$ is equal to 0 or -1 \cite[Chapter 2]{Kaw96}
so for our purposes, we can assume that none of $p$, $q$, and $r$ are 0 or -1.
Permuting the parameters $2p+1$, $2q+1$ and $2r+1$ yields the same (unoriented) knot.
Also, $P(-2p-1,-2q-1,-2r-1)$ and $P(2p+1,2q+1,2r+1)$ are mirrors of each other.
Since $\pi_1(M_J)$ is invariant of reversing orientation and mirroring,
we can assume that $1\leq q\leq r$.

\begin{thm} \label{mainthm}
Given integers $p$, $q$, and $r$ with $1\leq q \leq r$ and $p\neq0$ or $-1$, let $J$ be the $P(2p+1,2q+1,2r+1)$ pretzel knot with Alexander polynomial $\Delta_J$ whose leading coefficient is a prime power.
The commutator subgroup of $\pi_1(M_J)$ is residually torsion-free nilpotent if
\begin{itemize}
\item $p\geq 1$,
\item $J$ is $P(2p+1,3,2r+1)$,
\item $J$ is $P(-3,2q+1,2r+1)$ and $J$ is not $P(-3,5,5)$, $P(-3,5,7)$, $P(-3,5,9)$, $P(-3,5,11)$ or $P(-3,7,7)$, or
\item $J$ is $P(-5,2q+1,2r+1)$ and $J$ is  not
\begin{itemize}
\item $P(-5,7,R)$ when $R$ is 11, 13, 15, 17, 19, 21, 23 or 25,
\item $P(-5,9,R)$ when $R$ is 9, 11, 13, 15 or 17, or
\item $P(-5,11,R)$ when $R$ is 11 or 13.
\end{itemize}
\end{itemize}
\end{thm}
%$P(-5,7,11)$, $P(-5,7,13)$, $P(-5,7,15)$, $P(-5,7,17)$, $P(-5,7,19)$, $P(-5,7,21)$,  $P(-5,7,23)$,  $P(-5,9,9)$, $P(-5,9,11)$, $P(-5,9,13)$ or $P(-5,11,11)$

\begin{remark}
Proposition \ref{trivial} is the only known obstruction to the commutator subgroup of a genus one pretzel knot group being residually torsion-free nilpotent
so the exceptional cases in Theorem \ref{mainthm} with nontrivial Alexander polynomial remain unresolved
and cannot be resolved with the technique used in this paper.
\end{remark}

When $p\leq-2$ and $1\leq q\leq r$, $P(2p+1,2q+1,2r+1)$ is not a pseudo-alternating knot; see Proposition \ref{pseudoaltprop}.
Therefore, all of the examples from Theorem \ref{mainthm}, where $p<-1$, are new examples of knots with residually torsion-free nilpotent commutator subgroups. 

In addition, we also obtain pretzel knots of arbitrarily high genus whose groups have residually torsion-free nilpotent commutator subgroups.
However, we were not able to determine whether or not these knots are pseudo-alternating
so it is possible this result follows from Mayland and Murasugi's work.

\begin{thm} \label{mainthm2}
If $J$ is a $P(3,-3,\ldots,3, -3, 2r+1)$ pretzel knot for some integer $r$, then the commutator subgroup of $\pi_1(M_J)$ is residually-torsion free nilpotent.
\end{thm}

\subsection{Possible Generalizations}

The techniques used here have a few limitations.
First, while our method can be applied to many families of genus one pretzel knots on a case by case basis,
this method does not lend itself well to generalizing to all genus one pretzel knots since many of the details depend on the arithmetic properties of $p$, $q$ and $r$.
Secondly, Mayland's method requires a couple conditions (an unknotted Seifert surface satisfying the free factor property and an Alexander polynomial with prime power leading coefficient) which may not be necessary for a knot group to have residually torsion-free nilpotent commutator subgroup.
Nevertheless, we make the following prediction for genus one pretzel knots.

\begin{conj} \label{genconj}
If $J$ is a genus one pretzel knot then the commutator subgroup of $\pi_1(M_J)$ is residually torsion-free nilpotent if and only if the Alexander polynomial of $J$ is nontrivial.
\end{conj}

\subsection{Application to Bi-Orderability}

A group is said to be \emph{bi-orderable} if there exists a total order of the group's elements invariant under both left and right multiplication.
The following fact \cite[Theorem B]{CGW15} follows from work of Linnell, Rhemtulla, and Rolfsen \cite{LRR08}
and has been instrumental in determining the bi-orderability of several knot groups \cite{PerRolf03,CDN16,John20a}.

\begin{thm}{\cite[Theorem B]{CGW15}} \label{thmknotbo}
Let $J$ be a knot in $S^3$.
If $\pi_1(M_J)$ has residually torsion-free nilpotent commutator subgroup and all the roots of $\Delta_J$ are real and positive then $\pi_1(M_J)$ is bi-orderable.
\end{thm}

Furthermore, Ito obtained the following obstruction to a knot group being bi-orderable when the knot is rationally homologically fibered \cite{Ito16};
see section \ref{secmayland} for definition of rationally homologically fibered.

\begin{thm}{\cite[Theorem 2]{Ito16}} \label{thmboobst}
Let $J$ be a rationally homologically fibered knot.
If $\pi_1(M_J)$ is bi-orderable then $\Delta_J$ has at least one real positive root.
\end{thm}

The Alexander polynomial of the pretzel knot $P(2p+1,2q+1,2r+1)$ has the following form; see section \ref{secgenusone} for details.
\[
\Delta_J(t)=Nt^2+(1-2N)t+N
\]
where
\begin{equation} \label{detn}
N=\det
\left(
\begin{array}{cc}
p+q+1 & -q-1 \\
-q & q+r+1
\end{array}
\right).
\end{equation}
$\Delta_J$ has two positive real roots when $N<0$
and two non-real roots when $N>0$.
When $N=0$, $\Delta_J(t)=1$.
Therefore, we have the following proposition.
\begin{prop} \label{propnbo}
Let $J$ be the $P(2p+1,2q+1,2r+1)$ pretzel knot,
and let $N$ be defined as in (\ref{detn}).
When the commutator subgroup of $\pi_1(M_J)$ is residually torsion-free nilpotent and $N<0$, $\pi_1(M_J)$ is bi-orderable.
When $N>0$, $\pi_1(M_J)$ is never bi-orderable, regardless of whether or not the commutator subgroup of $\pi_1(M_J)$ is residually torsion-free nilpotent.
\end{prop}

Applying Proposition \ref{propnbo} to the results in Theorem \ref{mainthm} yields the following corollary.

\begin{cor}
Given integers $p$, $g$ and $r$ with $1\leq q \leq r$ and $p\neq0$ or $-1$, let $J$ be the $P(2p+1,2q+1,2r+1)$ pretzel knot with Alexander polynomial $\Delta_J$.
\begin{enumerate}
\item $\pi_1(M_J)$ is bi-orderable if
\begin{itemize}
\item $J$ is $P(-3,3,2r+1)$,
\item $J$ is $P(-5,3,2r+1)$ and $|\Delta_J(0)|$ is a prime power, or
\item $J$ is $P(-5,7,7)$ or $P(-5,7,9)$.
\end{itemize}
\item $\pi_1(M_J)$ is not bi-orderable if
\begin{itemize}
\item $p \geq 1$,
\item $J$ is $P(-3,5,2r+1)$ with $r>3$,
\item $J$ is $P(-3,2q+1,2r+1)$ with $q\geq 2$,
\item $J$ is $P(-5,7,2r+1)$ with $r\geq 9$,
\item $J$ is $P(-5,9,2r+1)$ with $r\geq 6$, or
\item $J$ is $P(-5,2q+1,2r+1)$ with $q\geq 5$.
\end{itemize}
\end{enumerate}
\end{cor}

We also have the following corollary to Theorem \ref{mainthm2}.

\begin{cor} \label{maincor2}
If $J$ is the $P(3,-3,\ldots,3, -3, 2r+1)$ pretzel knot for some integer $r$, then $\pi_1(M_J)$ is bi-orderable.
\end{cor}

Details of the proof of Corollary \ref{maincor2} are provided in section \ref{sechighergenus}.

\subsection{A Possible Connection of Bi-Orderability to Branched Covers}

Given a knot $J$ in $S^3$, let $\Sigma_n(J)$ be the \emph{$n$-fold cyclic cover of $S^3$ branched over $J$};
see \cite[Chapter 10]{Rolf76} for the definition and construction of a cyclic branched cover.
Part of the motivation for studying the bi-orderability of pretzel knots is to investigate the following questions.

\begin{quest} \label{quest1}
Do there exist knots with $\pi_1(M_J)$ bi-orderable and $\pi_1(\Sigma_n(J))$ left-orderable for some $n$?
\end{quest}

\begin{quest} \label{quest2}
Does $\pi_1(M_J)$ not being bi-orderable imply that $\pi_1(\Sigma_n(J))$ is left-orderable for some $n$?
\end{quest}
 
Question \ref{quest1} is resolved here.

\begin{thm} \label{countexamthm}
For each integer $q\geq 3$, let $J_q$ be the $P(1-2q,2q+1,4q-3)$ pretzel knot.
When $q-1$ is a prime power, $\pi_1(M_{J_q})$ is bi-orderable, and
$\pi_1(\Sigma_2(J_q))$ is left-orderable.
\end{thm}

\begin{remark}
Question \ref{quest1} is still unanswered for fibered knots and alternating knots.
\end{remark}

Question \ref{quest2} remains unresolved as of the writing of this paper.
However, some important remarks can be made about this question.

Suppose $J$ is a pretzel knot $P(2p+1,2q+1,2r+1)$ with $1\leq q\leq r$.
When $p\geq 1$, the signature of $J$ is nonzero
which likely means that $\pi_1(\Sigma_n(J))$ is left-orderable for $n$ sufficiently large; see \cite[Corollary 1.2, Question 1.3]{Gor17}.

Suppose $p<-1$.
By the Montesinos trick \cite{Mon75}, the double branched cover of $J$ is the Seifert fibered space
\[
\Sigma_2(J)=M(0;-1,\frac{-2p-2}{-2p-1},\frac{1}{2q+1},\frac{1}{2r+1}).
\]
By work of Eisenbud-Hirsch-Neumann \cite{EHN81}, 
Lisca-Stipsicz \cite{LisSti09}, Jankins-Neuman \cite{JanNeu85}, Naimi \cite{Nai94}, and Boyer-Rolfsen-Weist \cite{BRW05},
$\Sigma_2(J)$ is left-orderable if and only if there are positive integers $a$ and $m$ such that the triple $(\frac{-2p-2}{-2p-1},\frac{1}{2q+1},\frac{1}{2r+1})$ is less than some permutation of the triple $(\frac{a}{m},\frac{m-a}{m},\frac{1}{m})$.
This happens precisely when $1<-p\leq q$.
In this case, $m=2q$ and $a=2q-1$.
Therefore, we can state the following proposition.

\begin{prop} \label{doublebranchprop}
Suppose $J$ is the $P(2p+1,2q+1,2r+1)$ pretzel knot with $p<-1$ and $1\leq q\leq r$.
$\pi_1(\Sigma_2(J))$ is left-orderable if and only if $-p\leq q$.
\end{prop}

Thus, if $p<-1$ and the double branched cover of $J$ does not have left-orderable fundamental group,
then $q<-p$
so $N$ as defined in (\ref{detn}) is negative.
Therefore, if Conjecture \ref{genconj} is true, $\pi_1(M_J)$ would be bi-orderable when $q<-p$ by Proposition \ref{propnbo}.
In particular, if Conjecture \ref{genconj} is true, it's not likely that any non-alternating genus one pretzel knot would be a counterexamples for Question \ref{quest2}.

There is some evidence that genus one pretzel knots with no left-orderable cyclic branched covers do exists.
It is conjectured \cite{BGW13} that given a prime orientable closed rational homology sphere $Y$,
$\pi_1(Y)$ is not left-orderable if and only if $Y$ is an L-space,
and Issa and Turner show that the cyclic branched covers of the $P(-3,3,2r+1)$ pretzel knots are all L-spaces \cite{IssTur20}.

\subsection{Outline}
In section \ref{secmayland}, we review how Mayland's technique \cite{May75} can be used to analyze when the commutator subgroup of a knot group is residually torsion-free nilpotent.
In section \ref{secgenusone}, we apply this technique to genus one pretzel knots and prove Theorem \ref{mainthm} and Theorem \ref{countexamthm}.
In section \ref{sechighergenus}, we prove Theorem \ref{mainthm2}.
Appendix \ref{secproofs} contains the proofs of some key lemmas.
We also provide a chart of our results in appendix \ref{seccharts}.

\subsection{Acknowledgments}

The author would like to thank Cameron Gordon for his guidance and encouragement throughout this project.
The author would like to thank Ahmad Issa and Hannah Turner for many helpful conversations. The author would like to thank the anonymous Algebraic \& Geometric Topology referee for helpful comments and critiques. This research was supported in part by NSF grant DMS-1937215.

%%%%%%%%%%%%%%%%%%%%%%%%%%%%%%%%%%%%%%%%%%%%%%%%%%%%%%%%%%%%%%%%%%%%%%
\section{Preliminaries on Mayland's Technique} \label{secmayland}
%%%%%%%%%%%%%%%%%%%%%%%%%%%%%%%%%%%%%%%%%%%%%%%%%%%%%%%%%%%%%%%%%%%%%%

Mayland used a description of the commutator subgroup of a knot group to investigate when they are residual finite \cite{May75}.
In this section, we show how Mayland's technique can be used to find a sufficient condition for the commutator subgroup of a knot group to be residually torsion-free nilpotent.

\subsection{Mayland's Technique}

Let $J$ be a knot in $S^3$, and suppose $J$ bounds a minimal genus Seifert surface $S$ such that $S$ is \emph{unknotted}, in other words, $\pi_1(S^3\backslash S)$ is a free group.
Let $\hat{S}=M_J\cap S$.
Let $G$ be the commutator subgroup of $\pi_1(M_J)$.

Let $U$ be the image of a bi-collar embedding $\hat{S}\times[-1,1]\hookrightarrow M_J$ where $\hat{S}$ is the image of $\hat{S}\times\{0\}$,
and let $M_S=M_J\backslash \hat{S}$.
Denote images of $\hat{S}\times (0,1]$ and $\hat{S}\times [-1,0)$ in $M_S$ as $U^+$ and $U^-$ respectively.
Let $X=\pi_1(M_S)$ which is a free group of rank $2g$ where $g$ is the genus of $J$.
Consider the inclusion maps $i^+:U^+\to M_S$ and $i^-:U^-\to M_S$.
Let $H$ be the image of the induced map $i^+_*:\pi_1(U^+)\to \pi_1(M_S)$ and $K$ be the image of $i^-_*:\pi_1(U^-)\to \pi_1(M_S)$.

For each integer $n$, let $X_n$ be a copy of $X$, $H_n\subset X_n$ be a copy of $H$,
and $K_n\subset X_n$ be a copy of $K$.
The fundamental groups of $U$, $U^+$ and $U^-$ are canonically isomorphic,
and since $S$ has minimal genus, $i^+_*$ and $i^-_*$ are injective.
Therefore, $H_n$ and $K_{n+1}$ are identified with a rank $2g$ free group $F$.
By a result of  Brown and Crowell \cite[Theorem 2.1]{BroCro66}, $G$ is an amalgamated free product of the following form.

\begin{equation} \label{cycliccoverprop}
G\cong \cdots \underset{F}{*} X_{-2} \underset{F}{*} X_{-1} \underset{F}{*} X_0 \underset{F}{*} X_1 \underset{F}{*} X_2 \underset{F}{*} \cdots
\end{equation}

Baumslag provides the following sufficient condition \cite[Proposition 2.1(i)]{Baum69} for a group to be residually torsion-free nilpotent when $G$ is an ascending chain of \emph{parafree} subgroups; see \cite{Baum67,Baum69} for definition and discussion of parafree groups.
\begin{prop}{\cite[Proposition 2.1(i)]{Baum69}} \label{baumprop}
Suppose $G$ is a group which is the union of an ascending chain of groups as follows.
$$G_0<G_1<G_2<\cdots<G_n<\cdots<G=\bigcup_{n=1}^{\infty}G_n$$
Suppose each $G_n$ is parafree of the same rank.
If for each non-negative integer $n$, $|G_{n+1}:G_n[G_{n+1},G_{n+1}]|$ is finite, then
$G$ is residually torsion-free nilpotent.
\end{prop}

For each non-negative integer $m$, define $Z^m$ as follows.
\begin{equation} \label{zs}
Z^m:=X_{-m}\underset{F}{*}X_{1-m}\underset{F}{*}\cdots\underset{F}{*}X_{m-1}\underset{F}{*}X_m
\end{equation}
The direct limit of the $Z^m$'s is isomorphic to $G$.
Furthermore, since $i^+_*$ and $i^-_*$ are injective, the natural inclusion $Z^m\hookrightarrow Z^{m+1}$ is an embedding so $G$ is an ascending chain of subgroups as follows.
\[
Z^0<Z^1<Z^2<\cdots<Z^m<\cdots<G=\bigcup_{m=1}^{\infty}Z^m
\]

A subgroup $A$ of a free group $B$ is a \emph{free factor} if $B=A*D$ for some subgroup $D$ of $B$.
It immediately follows that $A$ is a free factor of $B$ if and only if every (equivalently, at least one) free basis of $A$ extends to a free basis of $B$.
A theorem of Mayland \cite[Theorem 3.2]{May75} provides the following sufficient conditions for each $Z^m$ to be parafree.

\begin{prop}{\cite[Theorem 3.2]{May75}} \label{mayprop}
If $H$ and $K$ are free factors of $H [X,X]$ and $K [X,X]$ respectively, and if additionally, $|X:H[X,X]|=|X:K[X,X]|=p^l$ for some prime $p$ and non-negative integer $l$, then for every non-negative $m$, $Z^m$ is parafree of rank $2g$.
\end{prop}

The knot $J$ is \emph{rationally homologically fibered} if the induced map on homology, $i^+_h:H_1(U^+;\Q)\to H_1(M_S;\Q)$ (or equivalently $i^-_h:H_1(U^-;\Q)\to H_1(M_S;\Q)$), is an isomorphism.
Let $S_+$ be a Seifert matrix representing $i^+_h$ so that $S_-:=S_+^T$ is a Seifert matrix representing $i^-_h$.
$S_+$ is also a presentation matrix for the abelian group $X/H[X,X]$.
Similarly, $S_-$ is a presentation matrix for $X/K[X,X]$.
Thus,
\begin{equation} \label{HequalsK}
\frac{X}{H[X,X]}\cong\frac{X}{K[X,X]}.
\end{equation}

Denote the standard form of the Alexander polynomial of $J$ by $\Delta_J$.
For some non-negative integer $k$,
\[
t^k\Delta_J(t)=\det(tS_+-S_+^T)=d_0+d_1t+\cdots+d_{2g}t^{2g}.
\]
It is a well known fact that $d_i=d_{2g-i}$ (see \cite[Chapter 6]{Mur96}).

\begin{prop} \label{rhfprop}
Suppose $J$ is a knot in $S^3$.
The following statements are equivalent:
\begin{enumerate}[label=(\alph*)]
\item $J$ is rationally homologically fibered. \label{rhfrhf}
\item $|X:H[X,X]|$ is finite. \label{rhfh}
\item $|X:K[X,X]|$ is finite. \label{rhfk}
\item $\deg\Delta_J=2g$. \label{rhf2g}
\end{enumerate}
\end{prop}

\begin{proof}
The equivalence of \ref{rhfh} and \ref{rhfk} follows from (\ref{HequalsK}).

Since $S_+$ is a presentation matrix for $X/H[X,X]$,
$|X:H[X,X]|$ is finite if and only if $|\det(S_+)|\neq 0$.
It follows that \ref{rhfrhf} and \ref{rhfh} are equivalent.

Since $d_{2g}=d_0=\det(S_+)$, $\deg\Delta_J=2g$ if and only if $\det(S_+)\neq 0$ so \ref{rhfrhf} and \ref{rhf2g} are equivalent.
\end{proof}

\begin{prop} \label{leadcoefprop}
When $J$ is rationally homologically fibered,
\[
|X:H[X,X]|=|X:K[X,X]|=|\Delta_J(0)|.
\]
\end{prop}

\begin{proof}
When $J$ is rationally homologically fibered,
\[
|X:H[X,X]|=|\det(S_+)|=|\Delta_J(0)|
\]
so the proposition follows from (\ref{HequalsK}). 
\end{proof}

For each non-negative $m$,
\[
\frac{Z^{m+1}}{Z^m[Z^{m+1},Z^{m+1}]}\cong
\frac{X}{H[X,X]}\times\frac{X}{K[X,X]}
\]
so when $J$ is rationally homologically fibered,
\begin{equation} \label{quot_simp}
|Z^{m+1}:Z^m[Z^{m+1},Z^{m+1}]|=|X:H[X,X]||K:H[X,X]|=\Delta_J(0)^2
\end{equation}
by Proposition \ref{leadcoefprop}.

The Seifert surface $S$ is said to \emph{satisfy the free factor property} if $H$ and $K$ are free factors of $H[X,X]$ and $K[X,X]$ respectively.
Note that this property is independent of the orientation of $S$.
A sufficient condition for the residual torsion-free nilpotence of $G$ can be summarized as follows.

\begin{prop} \label{mayanaprop}
Suppose $J$ is a rationally homologically fibered knot in $S^3$ with unknotted minimum genus Seifert surface $S$. If $S$ satisfies the free factor property and $|\Delta_J(0)|$ is a prime power, then the commutator subgroup, $G$, is residually torsion-free nilpotent.
\end{prop}

\begin{proof}
Suppose $J$ is a rationally homologically fibered with unknotted minimum genus Seifert surface $S$ satisfying the free factor property,
and suppose $|\Delta_J(0)|$ is a prime power.

Define $Z^m$ for each non-negative integer $m$ as in (\ref{zs}).
By Proposition \ref{leadcoefprop},
$|X:H[X,X]|$ and $|K:H[X,X]|$ are prime powers since $J$ is rationally homologically fibered.
Thus, by Proposition \ref{mayprop}, each $Z^m$ is parafree of rank twice the genus of $J$.

By (\ref{quot_simp}), $|Z^{m+1}:Z^m[Z^{m+1},Z^{m+1}]|=\Delta_J(0)^2$ so $|Z^{m+1}:Z^m[Z^{m+1},Z^{m+1}]|$ is finite.
Therefore, by Proposition \ref{baumprop}, $G$ is residually torsion-free nilpotent.
\end{proof}

\subsection{Pseudo-Alternating Knots} \label{subsecpseudoalt}

A \emph{special alternating diagram} is an alternating link diagram in which all of the crossings have the same sign. 
Any link with such a diagram is called a \emph{special alternating link}.
The Seifert surface described by performing Seifert's algorithm on special alternating diagram is a \emph{primitive flat surface}.
A \emph{generalized flat surface} is any surface which can be obtained by combining some number of primitive flat surfaces by Murasugi sums.
A link which bounds a generalized flat surface is a \emph{pseudo-alternating link}.
Alternating links are pseudo-alternating links.
However, all torus links, many of which are not alternating, are also pseudo-alternating links.

Pseudo-alternating knots are rationally homologically fibered and bound surfaces satisfying the free factor condition \cite[Theorem 2.5]{MayMur76}.
Therefore, the knot group of a pseudo-alternating knot, whose Alexander polynomial has a prime power leading coefficient, has residually torsion-free nilpotent commutator subgroup.

%%%%%%%%%%%%%%%%%%%%%%%%%%%%%%%%%%%%%%%%%%%%%%%%%%%%%%%%%%%%%%%%%%%%%%
\section{Genus One Pretzel Knots} \label{secgenusone}
%%%%%%%%%%%%%%%%%%%%%%%%%%%%%%%%%%%%%%%%%%%%%%%%%%%%%%%%%%%%%%%%%%%%%%

Let $J$ be the $P(2p+1,2q+1,2r+1)$ pretzel knot for some integers $p$, $q$, and $r$ with $1\leq q\leq r$ and $p\neq -1$ or $0$.
Let $S$ be the unknotted genus 1 surface depicted in Figure \ref{figgenusone}, which we refer to as the \emph{standard Seifert surface} of $J$.
For the genus one pretzel knots which aren't two-bridge knots,
the standard Seifert surface is the unique Seifert surface of minimal genus up to isotopy and \cite{GodIsh06}.

\begin{figure}[b]
\includegraphics[scale=1.0]{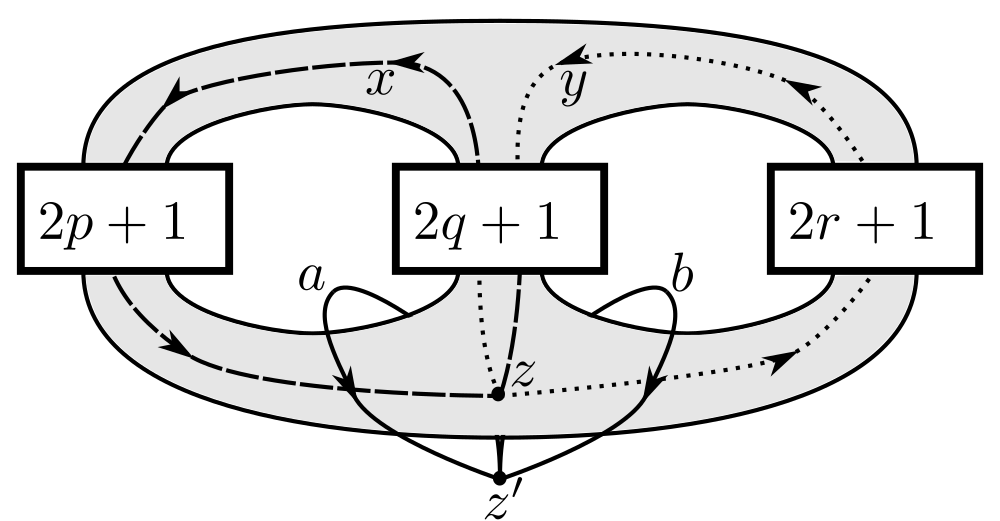}
\caption{The Seifert surface $S$ of $P(2p+1,2q+1,2r+1)$.}
\label{figgenusone}
\end{figure}

In this section, we show analyze when $S$ satisfies the free factor property. When $p>0$, $P(2p+1,2q+1,2r+1)$ is an alternating knot,
and thus, $P(2p+1,2q+1,2r+1)$ is pseudo-alternating.
However, this is not true when $p\leq-2$.

\begin{prop} \label{pseudoaltprop}
When $1\leq q\leq r$ and $p\leq-2$, the pretzel knot $P(2p+1,2q+1,2r+1)$ is not a pseudo-alternating knot.
\end{prop}
    
\begin{proof}
    Suppose $P(2p+1,2q+1,2r+1)$ is a pseudo-alternating knot.
    When $1\leq q\leq r$ and $p\leq-2$, the diagram given in Figure \ref{figpretzel} has a minimal number of crossings \cite[Theorem 10]{LicThis88}.
    Since this diagram is not alternating, $P(2p+1,2q+1,2r+1)$ cannot be alternating by a theorem of Kauffman, Murasugi, and Thistlethwaite \cite{Kauf87,Kauf88,Mur87,This87}.
    In particular, $P(2p+1,2q+1,2r+1)$ is not special alternating.
    Thus, $P(2p+1,2q+1,2r+1)$ must be the boundary of a surface $S$ which is the Murasugi sum of two generalized flat surfaces, $S_1$ and $S_2$, which aren't disk.

    By Gabai \cite{Gab85}, $S$ must be a minimal genus Seifert surface so $\chi(S)=-1$.
    Analyzing the effect of a Murasugi sum on the Euler characteristic yields
    \[
        -1=\chi(S)=\chi(S_1)+\chi(S_2)-1.
    \]
    Since $S_1$ and $S_2$ are not disks, neither $S_1$ or $S_2$ has positive Euler characteristic.
    It follows that $\chi(S_1)=\chi(S_2)=0$
    so $S_1$ and $S_2$ are both annuli.
    
    The boundary of a Murasugi sum of two annuli is a double twist knot which is alternating.
    Thus, $P(2p+1,2q+1,2r+1)$ is alternating, which is a contradiction.
\end{proof}

Since $J$ is pseudo-alternating when $p\geq 0$,
we will only need to focus on the case when $p$ is negative.

\subsection{Mayland's Technique for Genus One Pretzel Knots}

Define $M_J$, $M_S$, $X$, $H$ and $K$ as in section \ref{secmayland}.
Here we offer a concrete description of the maps on fundamental groups $i^+_*$ and $i^-_*$ for genus one pretzel knots.
This is the same discription used by Crowell and Trotter in \cite{CroTrot63}.
Choose a base point $z$ on the lower part of $S$, and let $x$ and $y$ be the classes generating $\pi_1(S,z)$ represented by the loops indicated in Figure \ref{figgenusone}.
Let $z^+$ and $z^-$ be push-offs of $z$ of each side of $S$.
Let $z'$ be the base point of $M_S$ obtained by shifting $z$ tangentially along $S$ through $\partial S$.
Let $\delta^+$ and $\delta^-$ be arcs connecting $z$ to $z^+$ and $z^-$ respectively; see Figure \ref{figbasepoints}.
Finally, let $a$ and $b$ be the indicated classes generating $\pi_1(M_S,z')$.

\begin{figure}[b]
\includegraphics[scale=1.0]{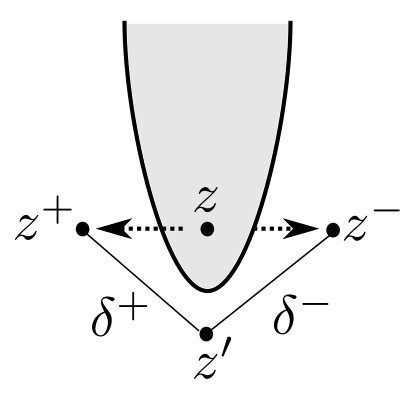}
\caption{Isotopy of base points}
\label{figbasepoints}
\end{figure}

By slightly isotoping elements of $\pi_1(S,z)$ off of $S$, $\pi_1(U^+,z^+)$ and $\pi_1(U^-,z^-)$ are canonically identified to $\pi_1(S,z)$ which is a rank two free group, $F$, generated by $x$ and $y$.
The group $X:=\pi_1(M_S,z')$ is a rank two free group generated by $a$ and $b$.
The map $i^+_*:F\to X$ takes a class $[\gamma]$ in $\pi_1(U^+,z^+)=F$ to the class $[\delta^+*\gamma*(-\delta^+)]$ in $\pi_1(M_S,z')=X$.
Likewise, the map $i^-_*:F\to X$ takes $[\gamma]$ to $[\delta^-*\gamma*(-\delta^-)]$.

With these choices, we define the elements
\begin{equation} \label{handk}\begin{array}{lcl}
\alpha_H:=i^+_*(x)=(b^{-1}a)^{q+1}a^{p} & \text{\hspace{1cm}} & \alpha_K:=i^-_*(x)=(ab^{-1})^{q}a^{p+1}\\
\beta_H:=i^+_*(y)=b^{r+1}(a^{-1}b)^{q} & & \beta_K:=i^-_*(y)=b^{r}(ba^{-1})^{q+1}
\end{array}
\end{equation}
so that
\[
H=\langle \{\alpha_H,\beta_H\}\rangle\text{\hspace{0.5cm}}K=\langle \{\alpha_K,\beta_K\}\rangle .
\]
Thus, the Seifert matrices for $i^+_*$ and $i^-_*$ are
\begin{equation} \label{seifmat}
S_+=\left(
\begin{array}{cc}
p+q+1 & -q-1\\
-q & q+r+1
\end{array}
\right)
\text{ and }
S_-=\left(
\begin{array}{cc}
p+q+1 & -q\\
-q-1 & q+r+1
\end{array}
\right)
.
\end{equation}

Let $N=\det S_+=\det S_-$.
Up to multiplication by a signed power of $t$, the Alexander polynomial of $J$ is
\[
\Delta_J(t)=Nt^2+(1-2N)t+N.
\]
When $N\neq0$, $J$ is rationally homologically fibered by Proposition \ref{rhfprop}.
Simply considering the integer $N$ can provide useful information.

\begin{prop} \label{rtfnobstructprop}
When $N=0$, $G$ is not residually torsion-free nilpotent.
\end{prop}

\begin{proof}
$\Delta_J(t)=1$ when $N=0$ so $G$ cannot be residually nilpotent by Proposition \ref{trivial}.
\end{proof}

\begin{prop} \label{ffpobstructprop}
If $|N|=1$, then the Standard Seifert surface $S$ does not satisfy the free factor property.
\end{prop}

\begin{proof}
Let $S$ be the standard Seifert surface of $J$, and
define $X$, $H$, and $K$ as in section \ref{secmayland}.
Each of these are rank 2 free groups.
Suppose $S$ satisfies the free factor property.

When $|N|=1$, $X/H[X,X]\cong X/K[X,X]\cong 1$ by Proposition \ref{leadcoefprop}
so $X=H[X,X]=K[X,X]$.
Since $H$ is a free factor of $H[X,X]$ and both are rank 2 free groups, $H=H[X,X]=X$.
Similarly, since $K$ is a free factor of $K[X,X]$ and both are rank 2 free groups, $K=X$.
This implies that $i^+_*$ and $i^-_*$ are isomorphisms.
Thus, $\pi_1(M_J)$ is an extension $\Z$ described by the following short exact sequence.
\[
1 \to X \to \pi_1(M_J) \to \Z \to 1
\]

The Stallings fibration theorem \cite{Stal61} implies that $J$ is a genus 1 fibered knot \cite{Stal61}.
However, the only genus 1 fibered knots are the trefoil and the figure 8 knot \cite{BurZie67,Gon70} which is a contradiction since we are assuming $J$ is not a two-bridge knot.
\end{proof}

In light of Proposition \ref{mayanaprop}, to prove the commutator subgroup of $\pi_1(M_J)$ is residually torsion-free nilpotent, it is sufficient to show $S$ satisfies the free factor property.

\subsection{Outline of Procedure}

In each case, we use the same basic procedure, outlined below, to analyze whether or not $S$ satisfies the free factor property.

\begin{enumerate}
\item Find a presentation matrix for $X/H[X,X]$ of the form
\[
\left(
\begin{array}{cc}
u & v \\
0 & w
\end{array}
\right)
\text{ or }
\left(
\begin{array}{cc}
u & 0 \\
v & w
\end{array}
\right)
\]
using row operations.
Note, $u$ and $w$ can always be made positive.
Thus, $X/H[X,X]$ is isomorphic to $(\Z/u\Z)\times(\Z/w\Z)$.
The $\Z/u\Z$ factor is generated by the class of $a$ in $X/H[X,X]$, and
the $\Z/w\Z$ factor is generated by the class of $b$.

\item Since $X/H[X,X]$ is abelian, the set $\mathcal{C}$, defined below, is a set of coset representatives of $H[X,X]$.
\[
\mathcal{C}=\{a^kb^l:0\leq k <u,0\leq l< w\}
\]
Given $x\in X$, denote by $\overline{x}$, the coset representative of $x$ in $\mathcal{C}$.
Define
\[
x_{c,x}:=cx(\overline{cx})^{-1}
\]
where $c\in \mathcal{C}$ and $x\in\{a,b\}$.
From this, we find the following free basis for $H[X,X]$ using the Reidemeister-Schreier method; see \cite{KMS66} for details.
\[
	\mathcal{B}=\{x_{c,x}:c\in \mathcal{C}, x\in\{a,b\},x_{c,x}\neq 1\}
\]

\item Use the Reidemeister-Schreier rewriting process to rewrite the generating set of $H$ from (\ref{handk}).
A word $\alpha\in H$, where $\alpha=\alpha_1^{s_1}\cdots \alpha_k^{s_k}$ with $\alpha_i\in\{a,b\}$ and $s_i=\pm 1$, can be rewritten as
\[
\alpha=x_{c_1,\alpha_1}^{s_1}\cdots x_{c_k,\alpha_k}^{s_k}
\]
where
\[
c_i=
\left\{
\begin{array}{ll}
\overline{\alpha_1\cdots \alpha_{i-1}} & \text{when }s_i=1\\
\overline{\alpha_1\cdots \alpha_{i}} & \text{when }s_i=-1
\end{array}
\right.
.
\]
\item Determine if the generating set of $H$ can be extended to a free basis of $H[X,X]$.

\item Repeat this procedure for $K$.
\end{enumerate}

When the free bases of $H$ and $K$ can be extended to free bases of $H[X,X]$ and $K[X,X]$ respectively,
$S$ satisfies the free factor property.
If the chosen basis of either $H$ or $K$ fails to extend then $S$ cannot satisfy the free factor property.

\subsection{Knots Where the Standard Seifert Surface Satisfies the Free Factor Property}

\begin{lem} \label{lemrandomcases}
If $J$ is $P(-5,7,7)$ or $P(-5,7,9)$ then $S$ satisfies the free factor property.
\end{lem}

\begin{proof}
Suppose $J$ is $P(-5,7,7)$.
From (\ref{handk}), we have that
\[
\begin{array}{lcl}
\alpha_H=(b^{-1}a)^4a^{-3} & \text{\hspace{1cm}} & \alpha_K=(ab^{-1})^{3}a^{-2}\\
\beta_H=b^4(a^{-1}b)^3 & & \beta_K=b^{3}(ba^{-1})^4
\end{array}.
\]

The abelian group $X/H[X,X]$ has presentation matrix
\[
\left(
\begin{array}{cc}
1 & -4 \\
-3 & 7
\end{array}
\right)
\]
which becomes
\[
\left(
\begin{array}{cc}
1 & 1 \\
0 & 5
\end{array}
\right)
\]
after row operations.

From this, we get $\mathcal{C}=\{1,b,b^2,b^3,b^4\}$ as a set of coset representatives of $H[X,X]$.
We apply Reidemeister-Schreier to obtain the following free basis of $H[X,X]$.
\[
	\mathcal{B}=\{ab,ba,b^2ab^{-1},b^3ab^{-2},b^4ab^{-3},b^5\}
\]
Label the basis elements as follows: $x_k:=b^kab^{1-k}$ for $0\leq k\leq 4$ and $x_5:=b^5$.

Now, we can rewrite $\alpha_H$ and $\beta_H$ in terms of $\mathcal{B}$.
\begin{align*}
	\alpha_H = &(b^{-5}) (b^4ab^{-3}) (b^2ab^{-1}) (ab) (b^{-5}) (b^3a^{-1}b^{-4}) (b^{5}) (b^{-1}a^{-1})\\
	=& x_5^{-1} x_4 x_2 x_0 x_5^{-1} x_4^{-1} x_5 x_0^{-1}
\end{align*}
and
\[
	\beta_H = (b^5) (b^{-1}a^{-1}) (ba^{-1}b^{-2}) (b^3a^{-1}b^{-4}) (b^5)=x_5 x_0^{-1} x_2^{-1} x_4^{-1} x_5.
\]
Thus,
\[
    \alpha_H = \beta_H^{-1} x_4^{-1} x_5 x_0^{-1}
\]
so
\[ 
    x_4=x_5 x_0^{-1} \alpha_H^{-1} \beta_H^{-1}
\]
and
\[
    x_2=\beta_H \alpha_H x_0 \beta_H^{-1} x_5 x_0^{-1}.
\]
Therefore, the set
\[
\{\alpha_H,\beta_H,x_0,x_1,x_3,x_5\}
\]
is a generating set of six elements for $H[X,X]$, and thus, is a free basis.
It follows that
\[
	H[X,X]=H*\{x_0,x_1,x_3,x_5\}
\]
so $H$ is a free factor of $H[X,X]$.

After row reductions, $X/K[X,X]$ has presentation matrix
\[
\left(
\begin{array}{cc}
1 & -3 \\
0 & 5
\end{array}
\right)
.
\]

From this we get a free basis of $K[X,X]$ as follows.
\[
x_k:=\left\{
\begin{array}{ll}
b^kab^{-3-k}& \text{for }0\leq k\leq 1 \\
b^kab^{2-k}& \text{for }2\leq k\leq 4 \\
b^5 & \text{for }k=5
\end{array} \right.
\]
Rewriting $\alpha_K$ and $\beta_K$, we get
%\beta_H=b^4(a^{-1}b)^3 & & \beta_K=b^{3}(ba^{-1})^4
\[
	\alpha_K = (ab^{-3}) (b^2a) (b^{-5}) (b^4ab^{-2}) (ba^{-1}b^{-3}) (b^3a^{-1})= x_0 x_2 x_5^{-1} x_4 x_3^{-1} x_0^{-1}
\]
and
\[
	\beta_K = (b^4a^{-1}b^{-1}) (b^2a^{-1}b^{-4}) (b^5) (a^{-1}b^{-2}) (b^3a^{-1}) =x_1^{-1} x_4^{-1} x_5 x_2^{-1} x_0^{-1}.
\]
Thus,
\[
    x_4=x_5 x_2^{-1} x_0^{-1} \beta_K^{-1} x_1^{-1}
\]
and
\[
    x_3= x_0^{-1} \alpha_K^{-1} \beta_K^{-1} x_1^{-1} .
\]

Therefore, the set
\[
\{\alpha_K,\beta_K,x_0,x_1,x_2,x_5\}
\]
is a free basis of $K[X,X]$ so $K$ is a free factor of $K[X,X]$.
Therefore, $S$ satisfies the free factor property.

Suppose $J$ is $P(-5,7,9)$.
$X/H[X,X]$ has presentation matrix
\[
\left(
\begin{array}{cc}
1 & -4 \\
-3 & 8
\end{array}
\right)
\]
which becomes
\[
\left(
\begin{array}{cc}
1 & 0 \\
0 & 4
\end{array}
\right)
\]
after row operations.

By Reidemeister-Schreier, we obtain the free basis $\{x_0,x_2,x_3,x_4\}$ where $x_i=b^iab^{-i}$ for $i=0,\ldots,3$ and $x_4=b^4$.

\begin{align*}
	\alpha_H = & (b^{-1}a)^4a^{-3}\\
	=& (b^{-4}) (b^3ab^{-3}) (b^2ab^{-2}) (bab^{-1}) (a^{-1}) (a^{-1})\\
	=& x_4^{-1} x_3 x_2 x_1 x_0^{-2} 
\end{align*}
and
\begin{align*}
	\beta_H = & b^4(a^{-1}b)^3\\
	=& (b^{-4}) (ba^{-1}b^{-1}) (b^2a^{-1}b^{-2}) (b^3a^{-1}b^{-3}) (b^4)\\
	=& x_4 x_1^{-1} x_2^{-1} x_3^{-1} x_4 .
\end{align*}
Thus,
\[ 
    x_4=\beta_H \alpha_H x_0^2
\]
and
\[
    x_3=x_4 \alpha_H x_0^2 x_1^{-1} x_2^{-1}.
\]
Therefore, the set
\[
\{\alpha_H,\beta_H,x_0,x_1,x_2\}
\]
is a free basis of $H[X,X]$ so
$H$ is a free factor of $H[X,X]$.

A similar argument shows $K$ is a free factor of $K[X,X]$.
Therefore, $S$ satisfies the free factor property.

\end{proof}

\begin{lem} \label{lem33r}
If $J$ is a $P(-3,3,2r+1)$ pretzel knot then $S$ satisfies the free factor property.
\end{lem}

\begin{proof}
From (\ref{handk}), we have that
\[
\begin{array}{lcl}
\alpha_H=b^{-1}ab^{-1}a^{-1} & \text{\hspace{1cm}} & \alpha_K=ab^{-1}a^{-1}\\
\beta_H=b^{r+1}a^{-1}b & & \beta_K=b^{r+1}a^{-1}ba^{-1}
\end{array}.
\]

The abelian group $X/H[X,X]$ has presentation matrix
\[
\left(
\begin{array}{cc}
1 & 0 \\
0 & 2
\end{array}
\right)
\]
when $r$ is even and
\[
\left(
\begin{array}{cc}
1 & -1 \\
0 & 2
\end{array}
\right)
\]
when $r$ is odd.

Using $\mathcal{C}=\{1,b\}$ as a set of coset representatives, we apply Reidemeister-Schreier to obtain a free basis of $H[X,X]$, $\mathcal{B}=\{x_0,x_1,x_2\}$

When $r$ is even,
\[
x_0=a, x_1=bab^{-1},\text{ and }x_2=b^2
\]
so
\[
	\alpha_H = (b^{-2})(bab^{-1})(a^{-1})=x_2^{-1}x_1x_0^{-1}
\]
and
\[
	\beta_H = (b^{2k})(ba^{-1}b^{-1})(b^2)=x_2^kx_1^{-1}x_2
\]
where $r=2k$.

When $r$ is odd,
\[
x_0=ab^{-1}, x_1=ba,\text{ and }x_2=b^2
\]
so
\[
	\alpha_H = (b^{-2})(ba)(b^{-2})(ba^{-1})=x_2^{-1}x_1x_2^{-1}x_0^{-1}
\]
and
\[
	\beta_H = (b^{2k+2})(a^{-1}b^{-1})(b^2)=x_2^{k+1}x_1^{-1}x_2
\]
where $r=2k+1$.

In either case, the set $\{\alpha_H,\beta_H,x_2\}$ is a free basis of $H[X,X]$ so $H$ is a free factor of $H[X,X]$.

$X/K[X,X]$ has presentation matrix
\[
\left(
\begin{array}{cc}
2 & 0 \\
0 & 1
\end{array}
\right)
.
\]

Using $\mathcal{C}=\{1,a\}$ as a set of coset representatives, we get the free basis of $K[X,X]$, $\mathcal{B}=\{x_0,x_1,x_2\}$ where
\[
x_0=a^2, x_1=b\text{ and }x_2=aba^{-1}.
\]
Thus,
\[
	\alpha_K = x_2^{-1}
\text{ and }
	\beta_K = x_1^{r+1}x_0^{-1}x_2.
\]
The set $\{\alpha_K,\beta_K,x_1\}$ is a free basis of $K[X,X]$ so $K$ is a free factor of $K[X,X]$.
Therefore, $S$ satisfies the free factor property.
\end{proof}

The proofs of the following results can be found in the appendix.

\begin{lem} \label{lemp3r}
If $J$ is a $P(2p+1,3,2r+1)$ pretzel knot with $p<-2$ then $S$ satisfies the free factor property.
\end{lem}

\begin{lem} \label{lem3qr}
Suppose $J$ is $P(-3,2q+1,2r+1)$ and one of the following conditions hold:
\begin{enumerate}
\item $q=2$ and $r\geq 6$,
\item $q=3$ and $r\geq 4$,
\item $q>3$.
\end{enumerate}
Then, $S$ satisfies the free factor property.
\end{lem}

\begin{lem} \label{lem5qr}
If $J$ is $P(-5,2q+1,2r+1)$ and one of the following conditions hold:
\begin{enumerate}
\item $q=3$ and $r\geq 13$,
\item $q=4$ and $r\geq 9$,
\item $q=5$ and $r\geq 7$,
\item $q>5$.
\end{enumerate}
Then, $S$ satisfies the free factor property.
\end{lem}

\subsection{Proof of Theorem \ref{countexamthm}}

For each integer $q\geq3$, let $J_q$ be the pretzel knot $P(1-2q,2q+1,4q-3)$ so $p=-q$ and $r=2q-2$

\begin{lem} \label{lemcounter}
For all $q\geq3$, the standard Seifert surface $S$ of $J_q$ satisfies the free factor property.
\end{lem}

\begin{proof}
The knot $J_3$ is $P(-5,7,9)$.
Thus, for $J_3$, $S$ satisfies the free factor property by Lemma \ref{lemrandomcases}.

Assume $q\geq 4$.
Define $X$, $H$, and $K$ as above.
After row reductions, $X/H[X,X]$ has presentation matrix
\[
\left(
\begin{array}{cc}
1 & -(q+1) \\
0 & -N
\end{array}
\right)
\]
where $N=-(q-1)^2$.

Let $C=-N=(q-1)^2$.
Using Reidemeister-Schreier, we obtain the following basis.
\[
    \{ab^{-q-1},bab^{-q-2},\ldots,b^{C-q-2}ab^{1-C},b^{C-q-1}a,b^{C-q}ab^{-1},\ldots,b^{C-1}ab^{-q},b^{C}\}
\]
To simplify computations, we modify this basis by multiplying some of the elements by $b^{-C}$ on the right, and obtain a free basis of $H[X,X]$, $\mathcal{B}=\{x_0,\ldots,x_{C}\}$
where $x_k=b^kab^{-q-1-k}$ for $k=0,\ldots,C-1$ and $x_{C}=b^{C}$.

We can rewrite $\alpha_H$ and $\beta_H$ as
\begin{align*}
	\alpha_H = & (b^{-1}a)^{q+1}a^{-q}\\ = & x_{C}^{-1} x_{C-1} x_{C} (x_{q-1} \cdots x_{i(q-2)-1} \cdots x_{q(q-2)-1}) x_{C} x_{q-2} x_{q-3}^{-1} x_{C}^{-1}\\
	& (x_{(q-3)(q+1)}^{-1} x_{(q-4)(q+1)}^{-1} \cdots x_{(q-i)(q+1)}^{-1} \cdots x_0^{-1})
\end{align*}
and
\begin{align*}
	\beta_H = & b^{2q-1}(a^{-1}b)^q\\
	= & x_{q-2}^{-1} x_{C}^{-1} (x_{q(q-2)-1}^{-1} \cdots x_{q(q-i)-1}^{-1} \cdots x_{q-1}^{-1}) x_{C}^{-1} x_{C-1}^{-1} x_{C}.
\end{align*}

Since $q\geq 4$, the generator $x_0$ appears once in the expression for $\alpha_H$ and does not appear in the expression for $\beta_H$.
Also, since $q-2<C-1$ and $qk-1<C-1$ for all $k=1,\ldots,q-2$, $x_{C-1}$ only appears once in the expression for $\beta_H$.

Thus, $x_{C-1}$ is a product of $\beta_H$, $x_1,\ldots,x_{C-2},x_C$ and
$x_0$ is a product of $\alpha_H$, $x_1,\ldots,x_C$.

Therefore, the set $\{\alpha_H,\beta_H,x_1,\ldots,x_{C-2},x_C\}$ is a free basis of $H[X,X]$ so $H$ is a free factor of $H[X,X]$.

After row reductions, $X/K[X,X]$ has presentation matrix
\[
\left(
\begin{array}{cc}
1 & -q \\
0 & C
\end{array}
\right) .
\]

We obtain a free basis of $K[X,X]$, $\mathcal{B}=\{x_0,\ldots,x_{C}\}$
where $x_k=b^kab^{-(q+k)}$ for $k=0,\ldots,C-1$ and $x_{C}=b^{C}$.

We can rewrite $\alpha_K$ and $\beta_K$ as
\begin{align*}
	\alpha_K = & (ab^{-1})^qa^{-q+1}\\
	= & (x_{0} x_{q-1} x_{2(q-1)} \cdots x_{(q-2)(q-1)}) x_{C} x_{0} x_{C}^{-1}\\
	& (x_{q(q-2)}^{-1} x_{q(q-3)}^{-1} \cdots x_0^{-1})
\end{align*}
and
\begin{align*}
	\beta_K = & b^{2q-2}(ba^{-1})^{q+1}\\
	= & x_{q-1}^{-1} x_0^{-1} x_{C}^{-1} (x_{(q-2)(q-1)}^{-1} x_{(q-3)(q-1)}^{-1} \cdots x_{0}^{-1}.
\end{align*}

The generator $x_{q}$ appears once in the expression for $\alpha_K$ and does not appear in the expression for $\beta_K$.
Also, $x_{C}$ only appears once in the expression for $\beta_K$.
Therefore, the set $\{\alpha_K,\beta_K,x_0,\ldots,x_{q-1},x_{q+1},\ldots, x_{C-1}\}$ is a free basis of $K[X,X]$ so $K$ is a free factor of $K[X,X]$.
Therefore, $S$ satisfies the free factor property.
\end{proof}

\begin{proof}[Proof of Theorem \ref{countexamthm}]
By Lemma \ref{lemcounter}, $J_q$ has a Seifert surface satisfying the free factor property.
The Alexander polynomial of $J_q$ is $Nt^2+(1-2N)t+N$ where $N=-(q-1)^2$
so $J_q$ is rationally homologically fibered and $\Delta_{J_q}$ has two positive real roots.

When $q-1$ is a prime power, $|\Delta_{J_q}(0)|=(q-1)^2$ is also a prime power.
Therefore, when $q-1$ is a prime power, $\pi_1(M_{J_q})$ has residually torsion-free nilpotent commutator subgroup by Proposition \ref{mayanaprop},
and $\pi_1(M_{J_q})$ is bi-orderable by Proposition \ref{propnbo}.
Since $p=-q$, $\Sigma_2(J_q)$ is left-orderable by Proposition \ref{doublebranchprop} for all $q\geq3$.
\end{proof}

\subsection{Knots Where the Standard Seifert Surface Does Not Satisfy the Free Factor Property}

\begin{lem} \label{lemmonic}
If $J$ is $P(1-2q,2q+1,2q^2+1)$ or $P(1-2q,2q+1,2q^2-3)$ then $S$ does not satisfy the free factor property.
\end{lem}

\begin{proof}
When $J$ is $P(1-2q,2q+1,2q^2+1)$, $p=-q$ and $r=q^2$.
When $J$ is $P(1-2q,2q+1,2q^2-3)$, $p=-q$ and $r=q^2-2$.
In both cases, $|N|=1$ so by Proposition \ref{ffpobstructprop}, $S$ does not satisfy the free factor property.
\end{proof}

\begin{lem} \label{lemjustdontwork}
Suppose $J$ is one of 
\begin{itemize}
\item $P(-3,5,11)$,
\item $P(-3,7,7)$,
\item $P(-5,7,R)$ for $R=11,13,21,23$ or $25$,
\item $P(-5,9,R)$ for $R=9,13,15$ or $17$,
\item $P(-5,11,11)$, or
\item $P(-5,11,13)$.
\end{itemize}
Then $S$ the standard Seifert surface of $J$ does not satisfy the free factor property.
\end{lem}

\begin{proof}

When $J$ is $P(-3,5,11)$, $X/H[X,X]$ has presentation matrix
\[
\left(
\begin{array}{cc}
1 & -1 \\
0 & 2
\end{array}
\right)
.
\]
We have the free basis $\mathcal{B}=\{ab^{-1},bab^{-2},b^2\}$ of $H[X,X]$.
Let $x_0=ab^{-1}$, $x_1=bab^{-2}$, and $x_2=b^2$
so
\[
\beta_H=b^6(a^{-1}b)^2=x_2^2 x_1^{-2} x_2.
\]

Let
\[
\Gamma:=\frac{H[X,X]}{\langle\beta_H^{H[X,X]}\rangle}\cong\langle x_0,x_1,x_2:x_2^3x_1^{-2}\rangle
\]
where $\langle\beta_H^{H[X,X]}\rangle$ is the normal closure of $\beta_H$ in $H[X,X]$.
Suppose $\{\alpha_H,\beta_H\}$ could be extended to a basis of $H[X,X]$,
then $\Gamma$ is a free group.
$\Gamma$ has as a subgroup isomorphic to $E:=\langle x_1,x_2:x_2^3x_1^{-2} \rangle$.
The abelianization of $E$ is $\Z$, but $E$ is not abelian since $x_1$ and $x_2$ do not commute.
Thus, $E$ is not free, and $\Gamma$ isn't free either, which is a contradiction.

Therefore, $H$ is not a free factor of $H[X,X]$,
and $S$ does not satisfy the free factor property.

When $J$ is $P(-5,7,25)$,
$H[X,X]$ has a free basis $x_0=a$, $x_1=bab^{-1}$, $x_2=b^2ab^{-2}$, $x_3=b^3ab^{-3}$, and $x_4=b^{4}$.
Under this basis
\[
    \beta_H\alpha_H=b^{12}a^{-2}=x_4^{3} x_0^{-2}.
\]

$\{\alpha_H,\beta_H\}$ can be extended to a free basis of $H[X,X]$ if and only if $\{\alpha_H,\beta_H\alpha_H\}$ can be extended to a free basis.
However, an argument similar to the previous case shows that $\beta_H\alpha_H$ cannot be extended to a basis of $H[X,X]$.

Therefore, $H$ is not a free factor of $H[X,X]$,
and $S$ does not satisfy the free factor property.

When $J$ is $P(-5,7,13)$ or $P(-5,7,21)$,
$H[X,X]$ has free basis $x_0=a$, $x_1=bab^{-1}$, and $x_2=b^{2}$.
Thus, the set $\{x_0,x_1,x_2^{-1} x_1 x_0\}$ is also a free basis of $H[X,X]$.
Denote $x_2^{-1} x_1 x_0$ by $y$.

Using the basis $\{x_0,x_1,y\}$,
\[
\alpha_H=(b^{-1}a)^4a^{-3}=y^2 x_0^{-3} .
\]
An argument similar to the previous cases shows that $\alpha_H$ cannot be extended to a basis of $H[X,X]$.
Therefore, $H$ is not a free factor of $H[X,X]$,
and $S$ does not satisfy the free factor property.

The proofs of the other cases are similar to the cases above.
Here we provide the elements obstructing the free factor property.

When $J$ is $P(-3,7,7)$,
\[
\beta_H=b^4(a^{-1}b)^3=x_2x_1^{-3}x_2
\]
where $x_0=ab^{-1}$, $x_1=bab^{-2}$, and $x_2=b^2$.

When $J$ is $P(-5,7,11)$,
\[
\beta_H=b^6(a^{-1}b)^3=x_3x_2^{-3}x_3
\]
where $x_0=ab^{-1}$, $x_1=bab^{-2}$, $x_2=b^2ab^{-3}$, and $x_3=b^3$.

When $J$ is $P(-5,7,23)$,
\[
\beta_H=b^12(a^{-1}b)^3=x_3^{3} x_2^{-3} x_3
\]
where $x_0=ab^{-1}$, $x_1=bab^{-2}$, $x_2=b^2ab^{-3}$, and $x_3=b^{3}$.

When $J$ is $P(-5,9,9)$,
\[
\beta_H=b^5(a^{-1}b)^4=x_0^{5} (x_2^{-1} x_1 x_0)^2
\]
where $x_0=b$, $x_1=aba^{-1}$, and $x_2=a^2$.

When $J$ is $P(-5,9,13)$,
\[
\beta_H=b^7(a^{-1}b)^4=x_0^{7} (x_2^{-1} x_1 x_0)^2
\]
where $x_0=b$, $x_1=aba^{-1}$, and $x_2=a^2$.

When $J$ is $P(-5,9,15)$,
\[
\beta_K\alpha_K=b^8a^{-3}=x_4^2 x_0^{-3}
\]
where $x_0=a$, $x_1=bab^{-1}$, $x_2=b^2ab^{-2}$, $x_3=b^3ab^{-3}$, and $x_4=b^4$.

When $J$ is $P(-5,9,17)$,
\[
\beta_H=b^9(a^{-1}b)^4=(x_0 x_6^{-1} x_4 x_2)^3 (x_6^{-1} x_5 x_2)^2
\]
where $x_0=ba^{2}$, $x_1=aba$, $x_2=a^2b$, $x_3=a^3ba^{-1}$, $x_4=a^4ba^{-2}$, $x_5=a^5ba^{-3}$ and $x_6=a^6$.

When $J$ is $P(-5,11,11)$,
\[
\beta_H=b^6(a^{-1}b)^5=x_3 x_2^{-5} x_3
\]
where $x_0=ab^{-1}$, $x_1=bab^{-2}$, $x_2=b^2ab^{-3}$, and $x_3=b^{3}$.

When $J$ is $P(-5,11,13)$,
\[
\beta_K=b^6(a^{-1}b)^6=(x_0 x_3 x_6)^3(x_0 x_2 x_4 x_6)^2
\]
where $x_0=ba^{-3}$, $x_1=aba^{-4}$, $x_2=a^2ba^{-5}$, $x_3=a^3ba^{-6}$, $x_4=a^4ba^{-7}$, $x_5=a^5ba^{-8}$, and $x_6=a^{6}$.
\end{proof}

\subsection{Proof of Theorem \ref{mainthm}}

\begin{lem} \label{lemtrivialcases}
If $J$ is $P(-3,5,7)$, $P(-5,7,17)$ or $P(-5,9,11)$  then $\pi_1(M_J)$ does not have residually torsion-free nilpotent commutator subgroup.
\end{lem}

\begin{proof}
For each of these knots, $N=0$ so this follows from Proposition \ref{rtfnobstructprop}.
\end{proof}

\begin{proof}[Proof of Theorem \ref{mainthm}]
When $p\geq 1$, $S$ is pseudo-alternating so $S$ satisfies the free factor property \cite{MayMur76}.
Therefore, when $p\geq 1$, the knot group of $P(2p+1,2q+1,2r+1)$ has residually torsion-free nilpotent commutator subgroups when $|\Delta_J(0)|$ is a prime power.

The other positive results follow from applying Proposition \ref{mayanaprop} to Lemma \ref{lemrandomcases}, Lemma \ref{lem33r}, Lemma \ref{lemp3r}, Lemma \ref{lem3qr}, Lemma \ref{lem5qr},
and Lemma \ref{lemtrivialcases}.
\end{proof}

%%%%%%%%%%%%%%%%%%%%%%%%%%%%%%%%%%%%%%%%%%%%%%%%%%%%%%%%%%%%%%%%%%%%%%
\section{Higher Genus Pretzel Knots} \label{sechighergenus}
%%%%%%%%%%%%%%%%%%%%%%%%%%%%%%%%%%%%%%%%%%%%%%%%%%%%%%%%%%%%%%%%%%%%%%

In this section, we prove Theorem \ref{mainthm2} which presents a family of pretzel knots with arbitrarily high genus whose groups have residually torsion-free nilpotent commutator subgroups.

Let $k$ be a positive integer, and let $r$ be any integer.
Suppose $J$ is the $2k+1$ parameter pretzel knot $P(3,-3,\ldots,3, -3, 2r+1)$ with genus $k$ Seifert surface $S$ as shown in Figure \ref{fighighergenus}.
Define $X$, $H$ and $K$ as in section \ref{secmayland}.

\begin{figure}[b]
\includegraphics[scale=1.0]{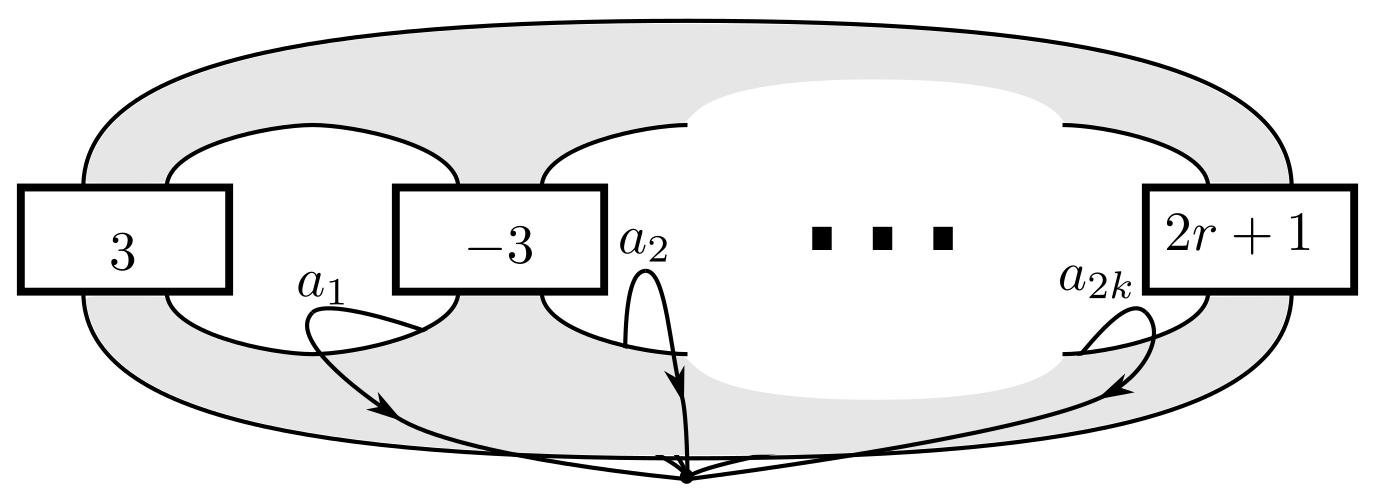}
\caption{Seifert surface for higher genus pretzel knots}
\label{fighighergenus}
\end{figure}

\begin{proof}[Proof of Theorem \ref{mainthm2}]
$X$ is a free group of rank $2k$ with generating set $\{a_1,\ldots,a_{2k}\}$ as show in Figure \ref{fighighergenus}.
By choosing a suitable free basis for $\pi_1(S)$, the subgroup $H$ has the following free basis.
\begin{align*}
\alpha_1= & (a_1^{-1}a_2)a_1 \\
\alpha_2= & (a_3^{-1}a_2)^2(a_2^{-1}a_1)^2\\
\vdots& \\
\alpha_{2i-1}= & (a_{2i-1}^{-1}a_{2i})(a_{2i-2}^{-1}a_{2i-1})\\
\alpha_{2i}= & (a_{2i+1}^{-1}a_{2i})^2(a_{2i}^{-1}a_{2i-1})^2 \\
\vdots& \\
\alpha_{2k-1}= & (a_{2k-1}^{-1}a_{2k})(a_{2k-2}^{-1}a_{2k-1})\\
\alpha_{2k}= & a_{2k}^{r+1}(a_{2k}^{-1}a_{2k-1})^2
\end{align*}
$X/H[X,X]$ has the presentation matrix
\begin{equation}\label{highgenusmatrix}
	\left(
	\begin{array}{cccccccc}
		0 & 1 & & & & & \\
		2 & 0 & -2 & & & & \\
		 & -1 & 0 & 1 & & & \\
		 & & 2 & 0 & -2 & & \\
		 & & & \ddots & \ddots & \ddots & \\
		 & & & & -1 & 0 & 1 \\
		 & & & & & 2 & r-1
	\end{array}
	\right)
\end{equation}
which after row operations becomes
\[
	\left(
	\begin{array}{cccccccc}
		0 & 1 & & & & & \\
		2 & 0 & 0 & & & & \\
		 & 0 & 0 & 1 & & & \\
		 & & 2 & 0 & 0 & & \\
		 & & & \ddots & \ddots & \ddots & \\
		 & & & & 0 & 0 & 1 \\
		 & & & & & 2 & 0
	\end{array}
	\right)
	.
\]

It follows that
\[
	\frac{X}{H[X,X]}\cong\bigoplus^k_{j=1}(\Z/2\Z)
\]
where the $j$th $\Z/2\Z$ factor is generated by the class of $a_{2j}$ in $X/H[X,X]$, and
when $i$ is odd, the class of $a_i$ is trivial.

Define
\[
a_{\bsig}:=a_1^{\sigma_1}a_3^{\sigma_2}\cdots a_{2k-1}^{\sigma_k}
\]
where $\bsig=(\sigma_1,\ldots,\sigma_k)\in\{0,1\}^k$.
$H[X,X]$ is an index $2^k$ subgroup of $X$ so the rank of $H[X,X]$ is $2^k+1$.

The following set is a set of coset representatives of $H[X,X]$.
\[
\mathcal{C}=\{a_{\bsig}:\bsig\in\{0,1\}^k\}
\]
From $\mathcal{C}$, we find a free basis $\mathcal{B}$ of elements of the form $x_{k,\bsig}:=a_{\bsig}a_k\overline{a_{\bsig}a_k}^{-1}$.

We point out a few important examples of basis elements.
For $i$ odd,
\[
	a_i^2=a_ia_i\overline{a_ia_i}^{-1}\in\mathcal{B}.
\]
For $i$ even,
\[
	a_i=1a_i\overline{1a_i}^{-1}\in\mathcal{B}.
\]
For $i$ odd and $j$ even,
\[
	a_ia_ja_i^{-1}=a_ia_j\overline{a_ia_j}^{-1}\in\mathcal{B}.
\]

Using the basis $\mathcal{B}$ rewrite the $\alpha_i$ as
\begin{align*}
\alpha_1= & (a_1^{-2})(a_1a_2a_1^{-1})(a_1^2) \\
\alpha_2= & (a_3^{-2})(a_3a_2a_3^{-1})(a_1a_2^{-1}a_1^{-1})(a_1^2)\\
\vdots& \\
\alpha_{2i-1}= & (a_{2i-1}^{-2})(a_{2i-1}a_{2i}a_{2i-1}^{-1})(a_{2i-1}a_{2i-2}^{-1}a_{2i-1}^{-1})(a_{2i-1}^2)\\
\alpha_{2i}= & (a_{2i+1}^{-2})(a_{2i+1}a_{2i}a_{2i+1}^{-1})(a_{2i-1}a_{2i}^{-1}a_{2i-1}^{-1})(a_{2i-1}^2) \\
\vdots& \\
\alpha_{2k-1}= & (a_{2k-1}^{-2})(a_{2k-1}a_{2k}a_{2k-1}^{-1})(a_{2k-1}a_{2k-2}^{-1}a_{2k-1}^{-1})(a_{2k-1}^2)\\
\alpha_{2k}= & a_{2k}^r(a_{2k-1}a_{2k}^{-1}a_{2k-1}^{-1})(a_{2k-1}^2)
\end{align*}
which can be extended to the free basis $\mathcal{B}'$ of $H[X,X]$ defined below.
\[
\mathcal{B}'=(\mathcal{B}-(\mathcal{B}_1\cup\mathcal{B}_2\cup\{a_{2k-1}^2\}))\cup\{\alpha_1,\ldots,\alpha_{2k}\}
\]
where
\[
\mathcal{B}_1=\{a_1a_2a_1^{-1},a_3a_4a_3^{-1},\ldots,a_{2k-1}a_{2k}a_{2k-1}^{-1}\}
\]
and
\[
\mathcal{B}_2=\{a_3a_2a_3^{-1},a_5a_4a_5^{-1},\ldots,a_{2k-1}a_{2k-2}a_{2k-1}^{-1}\}.
\]
Thus, $H$ is a free factor of $H[X,X]$.

A similar argument shows $K$ is a free factor of $K[X,X]$.
Thus, $S$ satisfies the free factor property.

From (\ref{highgenusmatrix}), we compute $|X:H[X,X]|=2^k$ so by Proposition \ref{rhfprop}, $J$ is rationally homologically fibered.
Thus, $S$ is an unknotted minimal genus Seifert surface, and $J$ is rationally homologically fibered.
It follows from Proposition \ref{mayanaprop} that the commutator subgroup of $J$ is residually torsion-free nilpotent.
\end{proof}

\begin{proof}[Proof of Corollary \ref{maincor2}]
From the Seifert matrix (\ref{highgenusmatrix}), we compute the following Alexander polynomial.
\[
\Delta_J(t)=(t-2)^k(2t-1)^k.
\]
It follows from Theorem \ref{mainthm2} and Theorem \ref{thmknotbo} that $\pi_1(M_J)$ is bi-orderable.
\end{proof}

%%%%%%%%%%%%%%%%%%%%%%%%%%%%%%%%%%%%%%%%%%%%%%%%%%%%%%%%%%%%%%%%%%%%%%
\appendix
%%%%%%%%%%%%%%%%%%%%%%%%%%%%%%%%%%%%%%%%%%%%%%%%%%%%%%%%%%%%%%%%%%%%%%

%%%%%%%%%%%%%%%%%%%%%%%%%%%%%%%%%%%%%%%%%%%%%%%%%%%%%%%%%%%%%%%%%%%%%%
\section{Proofs of Lemmas} \label{secproofs}
%%%%%%%%%%%%%%%%%%%%%%%%%%%%%%%%%%%%%%%%%%%%%%%%%%%%%%%%%%%%%%%%%%%%%%

In this appendix, we present the proofs of lemmas \ref{lemp3r}, \ref{lem3qr} and \ref{lem5qr}. Let $J$ be a pretzel knot $P(2p+1,2q+1,2r+1)$
with $1\leq q\leq r$.
Define the Seifert surface $S$ and
the groups $X\cong\langle a,b\rangle$, $H\cong\langle \alpha_H,\beta_H\rangle$, and $K\cong\langle \alpha_K,\beta_K\rangle$ as in section \ref{secgenusone}. 

\subsection{Proof of Lemma \ref{lemp3r}}

\begin{replem}{lemp3r}
If $J$ is a $P(2p+1,3,2r+1)$ pretzel knot with $p<-2$, then $S$ satisfies the free factor property.
\end{replem}
\begin{proof}
From (\ref{handk}), we have that
\[
\begin{array}{lcl}
\alpha_H=b^{-1}ab^{-1}a^{p+1} & \text{\hspace{1cm}} & \alpha_K=ab^{-1}a^{p+1}\\
\beta_H=b^{r+1}a^{-1}b & & \beta_K=b^{r+1}a^{-1}ba^{-1}
\end{array}.
\]

The abelian group $X/H[X,X]$ has a presentation matrix
\[
\left(
\begin{array}{cc}
1 & -r-2 \\
0 & -N
\end{array}
\right)
\]
where $N=pr+2p+2r+2=(p+2)(r+2)-2$, which is negative since $p\leq-2$.

Let $C=-N$.
Using $\mathcal{C}=\{1,b,\ldots,b^{C}\}$ as a set of coset representatives, we apply Reidemeister-Schreier to obtain a free basis of $H[X,X]$. Modifying this basis, we get
\[
\mathcal{B}=\{x_0,\ldots,x_{C}\}
\]
where $x_k:=b^kab^{-r-2-k}$ when $0\leq k\leq C-1$ and $x_{C}:=b^{C}$.

Using the rewriting process, we have that
\[
	\alpha_H = x_{C}^{-1} x_{C-1} (x_{C-r-4}^{-1} x_{C-2r-6}^{-1} \cdots x_{C-i(r+2)-2}^{-1} \cdots x_{r+2}^{-1} x_0^{-1})
\]
and
\[
	\beta_H = x_{C}^{-1} x_{C-1}^{-1} x_{C}.
\]
(Note that since $p<-2$, $C>r+2$ so $x_{r+2}$ is defined.)
We can extend $\{\alpha_H, \beta_H\}$ to the set $\{\alpha_H,\beta_H,x_1,\ldots,x_{C-2},x_{C}\}$ which is a free basis of $H[X,X]$ so $H$ is a free factor of $H[X,X]$.

$X/K[X,X]$ has a presentation matrix
\[
\left(
\begin{array}{cc}
-N & 0 \\
-p-2 & 1
\end{array}
\right) .
\]

Let $l=-p-2$ so $C=l(r+2)+2$.
Note that $l$ is a positive integer.
We obtain a free basis of $K[X,X]$.
\[
\mathcal{B}=\{x_0,\ldots,x_{C}\}
\]
where $x_k:=a^kba^{l-k}$ when $0\leq k\leq C-1$ and $x_{C}:=a^{C}$.

Using the rewriting process, we have that
\[
	\alpha_K = x_{l+1}^{-1}
\]
and
\[
	\beta_K = x_0 x_{C}^{-1} x_{l(r+1)+2} x_{lr+2} x_{l(r-1)+2} \cdots x_{2l+2} x_{l+1}.
\]
The set $\{\alpha_K,\beta_K,x_1,\ldots,x_l,x_{l+2},\ldots,x_{C}\}$ is a free basis of $K[X,X]$ so $K$ is a free factor of $K[X,X]$.
Thus, $S$ satisfies the free factor property.
\end{proof}

\subsection{Proof of Lemma \ref{lem3qr}}

\begin{replem}{lem3qr}
Suppose $J$ is $P(-3,2q+1,2r+1)$ and one of the following conditions hold:
\begin{enumerate}
\item $q=2$ and $r\geq 6$,
\item $q=3$ and $r\geq 4$,
\item $q>3$.
\end{enumerate}
Then, $S$ satisfies the free factor property.
\end{replem}

\begin{proof}
This lemma is shown by applying the outline from section \ref{secmayland} to two cases.
First, we address the case when $q=2$ and $r\geq 6$,
then we show the lemma is true when $q\geq 3$, $r\geq 4$ and $q\leq r$.
\bigskip{}

\noindent\textit{Case $q=2$ and $r\geq 6$:}
$X/H[X,X]$ has a presentation matrix
\[
\left(
\begin{array}{cc}
1 & -3 \\
0 & N
\end{array}
\right)
\]
where $N=r-3$.

$H[X,X]$ has free basis $x_k=b^kab^{-k-3}$ for $k=0,\ldots,N-1$, and $x_N=b^{N}$.
Under this basis
\[
    \alpha_H=(b^{-1}a)^2b^{-1}a^{-1}=x_N^{-1} x_{N-1} x_N x_1 x_0^{-1}
\]
and
\[
    \beta_H=b^{r+1}(ba^{-1})^2=x_N x_1^{-1} x_N x_{N-1}^{-1} x_N.
\]

Since $r\geq 6$, $N\geq 3$, so $x_{N-1}\neq x_1$.
Thus ,the set $\{\alpha_H,\beta_H,x_2,\ldots,x_N\}$ is a free basis of $H[X,X]$ so $H$ is a free factor of $H[X,X]$.

$X/K[X,X]$ has a presentation matrix
\[
\left(
\begin{array}{cc}
1 & -2 \\
0 & N
\end{array}
\right)
\]
where $N=r-3$.

$K[X,X]$ has free basis $x_k=b^kab^{-k-2}$ for $k=0,\ldots,N-1$, and $x_N=b^{N}$.
Under this basis
\[
    \alpha_K=(ab^{-1})^2a^{-1}=x_0 x_{1} x_0^{-1}
\]
and
\[
    \beta_K=b^{r+1}a^{-1}(ba^{-1})^2=x_N x_2^{-1} x_{1}^{-1} x_0^{-1}.
\]
The set $\{\alpha_K,\beta_K,x_0,x_3,\ldots,x_N\}$ is a free basis of $K[X,X]$ so $K$ is a free factor of $K[X,X]$.
\bigskip{}

\noindent\textit{Case $q\geq 3$ and $r\geq 4$:}
$X/H[X,X]$ has a presentation matrix
\[
\left(
\begin{array}{cc}
1 & -r \\
0 & N
\end{array}
\right)
\]
where $N=qr-q-r-1=(q-1)(r-1)-2$.
Note that since $q\geq 3$ and $r\geq 4$, $N>r-2>1$.

We then obtain a free basis $x_k=b^kab^{-r-k}$ for $k=0,\ldots,N-1$ and $x_N=b^N$.
Under this basis 
\begin{align*}
	\alpha_H = & (b^{-1}a)^{q+1}a^{-3}\\
	=& x_N^{-1} x_{N-1} x_N x_{r-2} x_{2r-3} \cdots x_{N-r+2} x_N x_1 x_0^{-1}
\end{align*}
and
\begin{align*}
	\beta_H =  & b^{r+1(a^{-1}b)^q}\\
	= & x_1^{-1} x_{N}^{-1} x_{N-r+2}^{-1} x_{N-2r+3}^{-1} \cdots x_{r-2} x_{N-1}.
\end{align*}
Since $N>r-2>1$, the set $\{\alpha_H,\beta_H,x_2,\ldots,x_{N}\}$ is a free basis of $H[X,X]$ so $H$ is a free factor of $H[X,X]$.

For $K$, we begin by substituting $a=a_*b_*$ and $b=b_*$ so that
\[
	\alpha_K = a_*^qb_*^{-1}a_*^{-1}\text{\hspace{0.5cm}}\beta_K=b_*^ra_*^{-q-1}.
\]
$X/K[X,X]$ has a presentation matrix
\[
\left(
\begin{array}{cc}
N & 0 \\
1-q & 1 
\end{array}
\right)
\]
where $N=qr-q-r-1$.

Under the basis $x_k=a_*b_*a_*^{1-q-k}$ for $k=0,\ldots,N-1$ and $x_N=a_*^N$,
\[
	\alpha_K = x_1^{-1}
\]
and
\[
	\beta_K = x_0 x_{q-1} \cdots x_{(q-1)(r-1)} x_{N}.
\]
Similarly to $H$, $K$ is a free factor of $K[X,X]$.
Thus, $S$ satisfies the free factor property.
\end{proof}

\subsection{Proof of Lemma \ref{lem5qr}}

\begin{replem}{lem5qr}
Suppose $J$ is $P(-5,2q+1,2r+1)$ and one of the following conditions hold:
\begin{enumerate}
\item $q=3$ and $r\geq 13$,
\item $q=4$ and $r\geq 9$,
\item $q=5$ and $r\geq 7$,
\item $q>5$.
\end{enumerate}
Then, $S$ satisfies the free factor property.
\end{replem}

This lemma is shown by applying the outline from section \ref{secmayland} to several cases.

\begin{lem}
If $J$ is $P(-5,7,2r+1)$ with $r\geq 13$,
then $S$ satisfies the free factor property.
\end{lem}
\begin{proof}
In this case, $q=3$ and $N=r-8$.
$X/H[X,X]$ has a presentation matrix
\[
\left(
\begin{array}{cc}
1 & -4 \\
0 & N
\end{array}
\right).
\]
We use the free basis, $x_k=b^kab^{-4-k}$ for $k=0,\ldots,N-1$ and $x_N=b^N$.

When $r=13$,
\[
	\alpha_H = x_5^{-1} x_{4} x_5 x_{2} x_{5} x_0 x_1^{-1} x_{5} x_0^{-1}
\]
and
\[
	\beta_H = x_5^2 x_{0}^{-1} x_5^{-1} x_{2}^{-1} x_{5}^{-1} x_{4}^{-1} x_{5}.
\]
The set $\{\alpha_H,\beta_H,x_2,\ldots,x_{N}\}$ is a free basis of $H[X,X]$ so $H$ is a free factor of $H[X,X]$.

When $r\geq 14$,
\[
	\alpha_H = x_N^{-1} x_{N-1} x_N x_{2} x_{5} x_{4}^{-1} x_0^{-1}
\]
and
\[
	\beta_H = x_N x_{5}^{-1} x_{2}^{-1} x_{N}^{-1} x_{N-1}^{-1} x_{N}.
\]
The set $\{\alpha_H,\beta_H,x_1,x_3,\ldots,x_{N}\}$ is a free basis of $H[X,X]$ so $H$ is a free factor of $H[X,X]$.

$X/K[X,X]$ has a presentation matrix
\[
\left(
\begin{array}{cc}
1 & -3 \\
0 & N
\end{array}
\right).
\]
We use the free basis $x_k=b^kab^{-3-k}$ for $k=0,\ldots,N-1$ and $x_N=b^N$.

Using this basis,
\[
	\alpha_K = x_0 x_2 x_4 x_{3}^{-1} x_0^{-1}.
\]
When $r=13$ or $r=14$,
\[
	\beta_K = x_N^2 x_{6-N}^{-1} x_N^{-1} x_{4}^{-1} x_{2}^{-1} x_{0}^{-1},
\]
and when $r\geq 15$,
\[
	\beta_K = x_N x_{6}^{-1} x_{4}^{-1} x_{2}^{-1} x_{0}^{-1}.
\]
In both cases, the set $\{\alpha_K,\beta_K,x_0,x_1,x_4,\ldots,x_{N}\}$ is a free basis of $K[X,X]$ so $K$ is a free factor of $K[X,X]$.
\end{proof}

\begin{lem}
If $J$ is $P(-5,7,2r+1)$ with $r\geq 9$,
then $S$ satisfies the free factor property.
\end{lem}
\begin{proof}
In this case, $q=4$ and $N=2r-10$.
$X/H[X,X]$ has a presentation matrix
\[
\left(
\begin{array}{cc}
N & 0 \\
-2 & 1
\end{array}
\right)
\]
after making the substitution $a=b_*^2a_*$ and $b=b_*$.
We use the free basis $x_k=a_*^kb_*a_*^{-2-k}$ for $k=0,\ldots,N-1$ and $x_N=b^N$.

When $r=9$,
\[
	\alpha_H = x_0 x_{3} x_6 x_8 x_{1} x_2^{-1} x_{8}^{-1} x_7^{-1} x_{5}^{-1} x_2^{-1} x_0^{-1}
\]
and
\[
	\beta_H = (x_0 x_{2} x_4 x_6 x_{8})^2 x_0 x_2 x_{1}^{-1} x_8^{-1} x_{6}^{-1} x_3^{-1} x_0^{-1}.
\]
The set $\{\alpha_H,\beta_H,x_0,x_1,x_2,x_4,x_6,x_7,x_8\}$ is a free basis of $H[X,X]$ so $H$ is a free factor of $H[X,X]$.

When $r=10$,
\[
	\alpha_H = x_0 x_{3} x_6 x_{9} x_{10} x_0^{-1} x_{10}^{-1} x_7^{-1} x_{5}^{-1} x_2^{-1} x_0^{-1}
\]
and
\[
	\beta_H = (x_0 x_{2} x_4 x_6 x_{8} x_{10})^2 x_0 x_{10}^{-1} x_9^{-1} x_{6}^{-1} x_3^{-1} x_0^{-1}.
\]

When $r\geq 11$,
\[
	\alpha_H = x_0 x_{3} x_6 x_{9} x_{10}^{-1} x_7^{-1} x_{5}^{-1} x_2^{-1} x_0^{-1}
\]
and
\[
	\beta_H = x_0 x_{2} \cdots x_{N-2} x_N x_0 x_{2} x_4 x_6 x_{8} x_{10} x_9^{-1} x_{6}^{-1} x_3^{-1} x_0^{-1}.
\]

In both cases, the set $\{\alpha_H,\beta_H,x_0,\ldots,x_6,x_8,x_{10},\ldots,x_{N}\}$ is a free basis of $H[X,X]$.

$X/K[X,X]$ has a presentation matrix
\[
\left(
\begin{array}{cc}
1 & 3-r \\
0 & N
\end{array}
\right) .
\]
We use the free basis $x_k=b^kab^{3-r-k}$ for $k=0,\ldots,N-1$ and $x_N=b^N$.
Using this basis,
\[
	\alpha_K = x_0 x_{r-4} x_N x_2 x_{r-2} x_{r-3}^{-1} x_0^{-1}
\]
and
\[
	\beta_K = x_4^{-1} x_N^{-1} x_{r-2}^{-1} x_2^{-1} x_{N}^{-1} x_{r-4}^{-1} x_{0}^{-1}.
\]
Since $r\geq 9$,
\[
    N=r-8+r-2>r-2>r-3>r-4>0
\]
so the generators $x_{r-2}$, $x_{r-3}$, and $x_{r-4}$ are valid generators.

The set $\{\alpha_K,\beta_K,x_1,\ldots,x_{r-4}, x_{r-2},\ldots,x_{N}\}$ is a free basis of $K[X,X]$ so $K$ is a free factor of $K[X,X]$.
\end{proof}

\begin{lem}
If $J$ is $P(-5,11,2r+1)$ with $r\geq 7$,
then $S$ satisfies the free factor property.
\end{lem}

\begin{proof}
In this case, $q=5$ $N=3r-12$.
$X/H[X,X]$ has a presentation matrix
\[
\left(
\begin{array}{cc}
1 & r-6 \\
0 & N
\end{array}
\right) .
\]
We use the free basis $x_k=b^kab^{r-6-k}$ for $k=0,\ldots,N-1$ and $x_N=b^N$.

Using this basis,
\[
	\alpha_H = x_{2r-6}^{-1} x_N x_0^{-1}.
\]
When $r=7$,
\[
	\beta_H = x_9 x_{0}^{-1} x_2^{-1} x_{4}^{-1} x_{6}^{-1} x_{8}^{-1} x_{9},
\]
and when $r\geq 8$,
\[
	\beta_H = x_{2r-5}^{-1} x_N x_2^{-1} x_{r-3}^{-1} x_{2r-8}^{-1} x_{3r-13}^{-1} x_{N}.
\]
Note that when $r\geq 8$,
\[
    N>3r-13>0
\]
\[
    N=r-7+2r-5>2r-5>2r-6>2r-8>0
\]
and
\[
    N=2r-9+r-3>r-3>0
\]
so the generators $x_{3r-13}$, $x_{2r-5}$, $x_{2r-6}$, $x_{2r-8}$, and $x_{r-3}$ are valid generators.

In both cases, the set $\{\alpha_H,\beta_H,x_1,x_{3},\ldots,x_{N}\}$ is a free basis of $H[X,X]$ so $H$ is a free factor of $H[X,X]$.

After making the substitution $a=b_*^2a_*$ and $b=b_*$, $X/K[X,X]$ has a presentation matrix
\[
\left(
\begin{array}{cc}
N & 0 \\
3 & 1
\end{array}
\right) .
\]
We use the free basis $x_k=a_*^kb_*a_*^{3-k}$ for $k=0,\ldots,N-1$ and $x_N=b^N$.
Using this basis,
\[
	\alpha_K = x_1 x_{N}^{-1} x_{N-1} x_{N-3} x_{N-5} x_{N-4}^{-1} x_{N-2}^{-1} x_{N} x_1^{-1}
\]
and
\[
	\beta_K = x_0 x_{N}^{-1} x_{N-3} x_{N-6} \cdots x_3 x_0 x_{N}^{-1} x_{N-3} x_{N-6}
	  x_{N-7}^{-1} x_{N-5}^{-1} x_{N-3}^{-1} x_{N-1}^{-1} x_{N} x_1^{-1}.
\]
Since $r\geq 7$, $N\geq 9$
so all the generators used are valid generators.

The set $\{\alpha_K,\beta_K,x_0,\ldots,x_{N-8}, x_{N-6},x_{N-5},x_{N-4},x_{N-3}, x_{N-1},x_{N}\}$ is a free basis of $K[X,X]$ so $K$ is a free factor of $K[X,X]$.
\end{proof}

\begin{lem}
If $J$ is $P(-5,13,13)$ or $P(-5,13,15)$,
then $S$ satisfies the free factor property.
\end{lem}
\begin{proof}
When $J$ is $P(-5,13,13)$, $p=-3$ and $q=r=6$.
$X/H[X,X]$ has presentation matrix
\[
\left(
\begin{array}{cc}
10 & 0 \\
-2 & 1
\end{array}
\right)
\]
We use the free basis $x_k:=a^kba^{-2-k}$ for $k=0\ldots 9$ and $x_{10}:=a^{10}$.

Using this basis,
\[
	\alpha_H = x_{10}^{-1} x_{8}^{-1} x_{7}^{-1} x_{6}^{-1} x_{5}^{-1} x_{4}^{-1} x_{3}^{-1} x_{2}^{-1}
\]
and
\[
	\beta_H = x_0 x_2 x_4 x_6 x_8 x_{10} x_0 x_2 x_3 x_4 x_5 x_6 x_7 x_8 x_{10}
\]
so
\[
    \beta_H\alpha_H=x_0 x_2 x_4 x_6 x_8 x_{10} x_0 .
\]
The set $\{\alpha_H,\beta_H,x_0,x_1, x_4, x_5,x_6,x_7,x_8, x_9 ,x_{10}\}$ is a free basis of $H[X,X]$ so $H$ is a free factor of $H[X,X]$.

$X/K[X,X]$ has presentation matrix
\[
\left(
\begin{array}{cc}
1 & 1 \\
0 & 10
\end{array}
\right)
\]
We use the free basis $x_k:=a^kba^{1-k}$ for $k=0\ldots 9$ and $x_{10}:=a^{10}$.

Using this basis,
\[
	\alpha_K = x_0 x_{10}^{-1} x_{8} x_{6} x_{4} x_{2} x_{0} x_{10}^{-1} x_{9}^{-1} x_{10} x_{0}^{-1}
\]
and
\[
	\beta_K = x_{8}^{-1} x_{10} x_0^{-1} x_2^{-1} x_4^{-1} x_6^{-1} x_8^{-1} x_{10} x_0^{-1} .
\]
The set $\{\alpha_K,\beta_K,x_0,x_1, x_3,x_4,x_5,x_6,x_7,x_8 ,x_{10}\}$ is a free basis of $K[X,X]$ so $K$ is a free factor of $K[X,X]$.
\bigskip{}

When $J$ is $P(-5,13,15)$, $p=-3$, $q=6$ and $r=7$.
After making the substitution $a=b_*^2a_*$ and $b=b_*$,
$X/H[X,X]$ has presentation matrix
\[
\left(
\begin{array}{cc}
14 & 0 \\
4 & 1
\end{array}
\right)
\]
We use the free basis $x_k:=a_*^kb_*a_*^{4-k}$ for $k=0\ldots 13$ and $x_{14}:=a^{14}$.

Using this basis,
\[
	\alpha_H = x_{0} x_{14}^{-1} x_{11} x_{8} x_{5} x_{2} x_{14}^{-1}  x_{13} x_{14}  x_{0}^{-1} x_{3}^{-1} x_{7}^{-1} x_{10}^{-1}  x_{14}  x_{0}^{-1}
\]
and
\[
	\beta_H = x_0  x_{14}^{-1} x_{10} x_{6} x_2 x_{14}^{-1} x_{12} x_8 x_4 x_0 x_{14}^{-1} x_{13}^{-1} x_{14} x_2^{-1} x_5^{-1} x_8^{-1} x_{11}^{-1} x_{14} x_{0}^{-1} .
\]
The set $\{\alpha_H,\beta_H,x_0,\ldots, x_5,x_8, \ldots, x_{14}\}$ is a free basis of $H[X,X]$ so $H$ is a free factor of $H[X,X]$.

$X/K[X,X]$ has presentation matrix
\[
\left(
\begin{array}{cc}
1 & 2 \\
0 & 14
\end{array}
\right)
\]
We use the free basis $x_k:=b^kab^{2-k}$ for $k=0\ldots 13$ and $x_{14}:=b^{14}$.

Using this basis,
\[
	\alpha_K = x_0 x_{14}^{-1} x_{11} x_{8} x_{5} x_{2} x_{14}^{-1} x_{13} x_{10} x_{14} x_{0}^{-1}
\]
and
\[
	\beta_K = x_{10}^{-1} x_{13}^{-1} x_{14} x_2^{-1} x_5^{-1} x_8^{-1} x_{11}^{-1} x_{14} x_0^{-1}
\]
so
\[
    \beta_H\alpha_H=x_{14} x_{0}^{-1} .
\]

The set $\{\alpha_K,\beta_K,x_1,x_3,\ldots ,x_{14}\}$ is a free basis of $K[X,X]$ so $K$ is a free factor of $K[X,X]$.
\end{proof}

\begin{lem} \label{lemheven}
If $J$ is $P(-5,2q+1,2r+1)$ with $q$ even, $q\geq 6$ and $r\geq 8$,
then $H$ is a free factor of $H[X,X]$.
\end{lem}
\begin{proof}
Let $c$ be the integer such that $q=2c$.
$X/H[X,X]$ has a presentation matrix
\[
\left(
\begin{array}{cc}
2 & -r \\
0 & w
\end{array}
\right)
\]
where $w=cr-2c-r-1$ and $N=2w$.

We have the following set of coset representatives.
\[
    \mathcal{C}=\{1,b,b^2,\ldots,b^{w-1},a,ab,ab^2,\ldots,ab^{w-1}\}
\]
We apply Reidemeister-Schreier to find a free basis of $H[X,X]$.
In the following computations, we assume that the coset representative of $a^2$, $\overline{a^2}$ is $b^r$.
For this to be correct, it must be true that $r<w$,
which we verify here.

Since $q\geq 6$, $c\geq 2$, and $r\geq 8$ so
\begin{align*}
    w=&cr-2c-r-1\\
    =&(c-3)(r-2)+(r-7)+r \\
    > &r .
\end{align*}

We apply Reidemeister-Schreier to find $x_{c,x}=cx(\overline{cx})^{-1}$ for each $c\in\mathcal{C}$ and $x\in\{a,b\}$.
\begin{align*}
    x_{b^i,a} = & b^ia(\overline{b^ia})^{-1}=
    \left\{
    \begin{array}{ll}
        b^iab^{-i}a^{-1}     & \text{if } 0 < i\leq w-1 \\
        1     & \text{if } i=0
    \end{array}
    \right. \\
    x_{b^i,b} = & b^{i+1}(\overline{b^{i+1}})^{-1}=
    \left\{
    \begin{array}{ll}
        1     & \text{if } 0 \leq i < w-1 \\
        b^w     & \text{if } i=w-1
    \end{array}
    \right. \\
    x_{ab^i,a} = & ab^ia(\overline{ab^ia})^{-1}=
    \left\{
    \begin{array}{ll}
        ab^iab^{-i-r}     & \text{if } 0 \leq i < w-r \\
        ab^iab^{w-i-r}     & \text{if } w-r \leq i \leq w-1
    \end{array}
    \right. \\
    x_{ab^i,b} = & ab^{i+1}(\overline{ab^{i+1}})^{-1}=
    \left\{
    \begin{array}{ll}
        1     & \text{if } 0 \leq i < w-1 \\
        ab^wa^{-1}     & \text{if } i=w-1
    \end{array}
    \right.
\end{align*}
The non-trivial elements $x_{c,x}$ form a basis $\{x_1, \ldots, x_{w},y_0,\ldots,y_{w}\}$ where
\[
    x_i=
    \left\{
    \begin{array}{ll}
        b^iab^{-i}a^{-1} & \text{if } 1\leq i \leq w-1 \\
        b^w & \text{if } i=w 
    \end{array}
    \right.
\]
and
\[
    y_i=
    \left\{
    \begin{array}{ll}
        ab^iab^{-i-r} & \text{if } 0 \leq i <w-r \\ 
        ab^iab^{w-i-r} & \text{if } w-r \leq i < w \\ 
        ab^wa^{-1} & \text{if } i=w
    \end{array}
    \right. .
\]

Using this basis,
\[
    \beta_H\alpha_H=y_0^{-1}
\]
and
\[
    \beta_H=y_{1}^{-1} x_{2}^{-1} \prod_{i=1}^{c-2} (y_{\delta(i)+1}^{-1} x_{\delta(i)+2}^{-1}) y_{w-2}^{-1} x_{w-1}^{-1} x_{w}
\]
where
\[
    \delta(i)=w-i(r-2) .
\]

We claim that for all $i$, $\delta(i)\neq0$.
Since $w=(c-1)(r-2)-3$,
\[
    \delta(i)=w-i(r-2)=(r-2)(c-i-1)-3
\]
so if $\delta(i)=0$ then $(r-2)(c-i-1)=3$.
However, since $r\geq 8$, $r-2$ does not divide 3.

Thus, $y_1$ only appears once in $\beta_H$
so the set $\{\beta_H\alpha_H,\beta_H,x_1,\ldots,x_{w},y_2,\ldots,y_{w}\}$ is a free basis of $H[X,X]$.
Since $\{\beta_H\alpha_H,\beta_H\}$ is a free basis of $H$,
$H$ is a free factor of $H[X,X]$.
\end{proof}

\begin{lem}\label{lemhodd}
If $J$ is $P(-5,2q+1,2r+1)$ with $q$ odd and $q\geq 7$,
then $H$ is a free factor of $H[X,X]$.
\end{lem}

\begin{proof}
Let $c$ be the integer such that $q=2c+1$.
$X/H[X,X]$ has a presentation matrix
\[
\left(
\begin{array}{cc}
1 & v \\
0 & N
\end{array}
\right)
\]
where $v=cr-2c-r-2$ and $N=2cr-4c-r-4=2v+r$.

We use the free basis $x_k=b^kab^{v-k}$ for $k=0,\ldots,N-1$ and $x_N=b^N$.
Using this basis,
\[
	\beta_H\alpha_H = x_{v+r}^{-1} x_{N} x_0^{-1}
\]
and
\[
	\beta_H = x_{v+r+1}^{-1}\prod_{i=0}^{2c-1} y_i
\]
where
\[
    y_i=\left\{
    \begin{array}{ll}
    x_{\epsilon(i)}^{-1}     & \text{if }\epsilon(i)<N-v-1 \\
    x_{\epsilon(i)}^{-1}x_N     & \text{if }\epsilon(i)\geq N-v-1 
    \end{array}
    \right.
\]
and 
\[
    \epsilon(i)=2+i(v+1) \mod N.
\]
Since $q\geq 7$, $c\geq 3$, and since $r\geq 7$,
\[
    v=cr-2c-r-2=(c-2)(r-2)+r-6>1
\]
This means that
\[
    N=2v+r>v+r+1>v+r>0
\]
so $x_{v+r}$ and $x_{v+r+1}$ are valid generators.

We claim that for each $i=0,\ldots,2c-1$, $\epsilon(i)$ is distinct.
Suppose that $\epsilon(i)=\epsilon(j)$ for some $i$ and $j$.
Then, $(j-i)(v+1)$ is a multiple of $N$.
In particular, $N$ divides $(j-i)\gcd(N,v+1)$.
Applying the Euclidean algorithm to $N$ and $v+1$,
we have that
\[
    N=2(v+1)+r-2
\]
and
\[
    v+1=(c-1)(r-2)-3
\]
so
\[
    \gcd(N,v+1)=\gcd(r-2,3)\leq 3.
\]
The maximum value of $j-i$ is $2c-1$.
It follows that
\[
    N\leq3(2c-1).
\]

However, since $c\geq 3$
and $r\geq 7$,
\begin{align*}
    N=& 2cr-4c-r-4\\
    =&(2c-1)(r-4)+4c-8\\
    \geq & 3(2c-1)+4>3(2c-1)
\end{align*}
which is a contradiction.

Thus $x_{\epsilon(0)}=x_2$ only appears once in $\beta_H$ so
the set $\{\beta_H\alpha_H,\beta_H,x_1,x_3,\ldots,x_{N}\}$ is a free basis of $H[X,X]$.
Therefore, $H$ is a free factor of $H[X,X]$.
\end{proof}

\begin{lem}\label{lemk0}
If $J$ is $P(-5,2q+1,2r+1)$ with $q\equiv 0 \mod 3$, $q\geq 6$, and $r\geq 8$,
then $K$ is a free factor of $K[X,X]$.
\end{lem}
\begin{proof}
Let $c$ be the integer such that $q=3c$.
$X/K[X,X]$ has a presentation matrix
\[
\left(
\begin{array}{cc}
1 & v \\
0 & N
\end{array}
\right)
\]
where $v=cr-2c-r-1$ and $N=3cr-6c-2r-2=3v+r+1$.

We use the free basis $x_k=b^kab^{v-k}$ for $k=0,\ldots,N-1$ and $x_N=b^N$.
Using this basis,
\[
	\beta_K\alpha_K = x_{v+r+1}^{-1} x_{2v+r+1}^{-1} x_{N} x_0^{-1}
\]
and
\[
	\beta_K = x_{v+r+1}^{-1} x_{2v+r+2}^{-1}\prod_{i=0}^{3c-2} y_i
\]
where
\[
    y_i=\left\{
    \begin{array}{ll}
    x_{\zeta(i)}^{-1}     & \text{if }\zeta(i)<N-v-1 \\
    x_{\zeta(i)}^{-1}x_N     & \text{if }\zeta(i)\geq N-v-1 
    \end{array}
    \right.
\]
and 
\[
    \zeta(i)=1+i(v+1) \mod N.
\]
Since $q\geq 6$, $c\geq 2$, and since $r\geq 8$,
\[
    v=cr-2c-r-1=(c-1)(r-5)+3c-6>1
\]
This means that
\[
    N=3v+r+1>2v+r+2>2v+r+1>v+r+1>0
\]
so $x_{v+r+1}$, $x_{2v+r+1}$ and $x_{2v+r+2}$ are valid generators.

Suppose that $\zeta(i)=\zeta(j)$ for some $i$ and $j$.
Then, $N$ divides $(j-i)\gcd(N,v+1)$.
Applying the Euclidean algorithm to $N$ and $v+1$,
we have that
\[
    N=3(v+1)+r-2
\]
and
\[
    v+1=(c-1)(r-2)-2
\]
so
\[
    N\leq2(3c-2).
\]

However, since $c\geq 2$
and $r\geq 8$,
\[
    N = 3cr-6c-2r-2 = (3c-2)(r-4)+6c-10 > 2(3c-2)
\]
so $\zeta(i)$ is distinct for each $i=0,\ldots,3c-2$.
Thus $x_{\zeta(0)}=x_1$ only appears once in $\beta_K$ so
the set $\{\beta_K\alpha_K,\beta_K,x_2,\ldots,x_{N}\}$ is a free basis of $K[X,X]$.
Therefore, $K$ is a free factor of $K[X,X]$.
\end{proof}

\begin{lem}\label{lemk1}
If $J$ is $P(-5,2q+1,2r+1)$ with $q\equiv 1 \mod 3$ and $q\geq 7$,
then $K$ is a free factor of $K[X,X]$.
\end{lem}
\begin{proof}
Let $c$ be the integer such that $q=3c+1$.
$X/K[X,X]$ has a presentation matrix
\[
\left(
\begin{array}{cc}
1 & -v \\
0 & N
\end{array}
\right)
\]
where $v=cr-2c-1$ and $N=3cr-6c-r-4=3v-r-1$.

We use the free basis $x_k=b^kab^{-v-k}$ for $k=0,\ldots,N-1$ and $x_N=b^N$.
Using this basis,
\[
	\beta_K\alpha_K = x_N^{-1} x_{2v}^{-1} x_{v}^{-1} x_0^{-1}
\]
and
\[
	\beta_K = x_N^{-1} x_{2v}^{-1} x_{v+1}^{-1}\prod_{i=0}^{3c-1} y_i
\]
where
\[
    y_i=\left\{
    \begin{array}{ll}
    x_{\eta(i)}^{-1}     & \text{if }\eta(i)<N-v+1 \\
    x_{x_N^{-1}\eta(i)}^{-1}     & \text{if }\eta(i)\geq N-v+1 
    \end{array}
    \right.
\]
and 
\[
    \eta(i)=2-i(v-1) \mod N.
\]
Since $q\geq 7$, $c\geq 2$, and since $r\geq 7$,
\[
    v=cr-2c-1=(c-1)(r-8)+4c-8+r+1>r+1
\]
This means that
\[
    N=3v-r-1>2v>v+1>v>0
\]
so $x_{v}$, $x_{v+1}$ and $x_{2v}$ are valid generators.

Suppose that $\eta(i)=\eta(j)$ for some $i$ and $j$.
Then, $N$ divides $(j-i)\gcd(N,v-1)$.
Applying the Euclidean algorithm to $N$ and $v-1$,
we have that
\[
    N=3(v-1)-(r-2)
\]
and
\[
    v-1=c(r-2)-2
\]
so
\[
    N\leq2(3c-1).
\]

However, since $c\geq 2$
and $r\geq 7$,
\[
    N = 3cr-6c-r-4 = (3c-1)(r-4)+6c-8 > 2(3c-1)
\]
so $\eta(i)$ is distinct for each $i=0,\ldots,3c-2$.
Thus $x_{\eta(0)}=x_2$ only appears once in $\beta_K$ so
the set $\{\beta_K\alpha_K,\beta_K,x_1,x_3,\ldots,x_{N}\}$ is a free basis of $K[X,X]$.
Therefore, $K$ is a free factor of $K[X,X]$.
\end{proof}

\begin{lem}\label{lemk2}
If $J$ is $P(-5,2q+1,2r+1)$ with $q\equiv 2 \mod 3$ and $q\geq 8$,
then $K$ is a free factor of $K[X,X]$.
\end{lem}
\begin{proof}
Let $c$ be the integer such that $q=3c+2$.
$X/K[X,X]$ has a presentation matrix
\[
\left(
\begin{array}{cc}
3 & -(r+1) \\
0 & w
\end{array}
\right)
\]
where $w=cr-2c-2$ and $N=3w$.

We have the following set of coset representatives.
\[
    \mathcal{C}=\{1,b,b^2,\ldots,b^{w-1},a,ab,\ldots,ab^{w-1}, a^2,a^2b,\ldots,a^2b^{w-1}\}
\]
We apply Reidemeister-Schreier to find a free basis of $K[X,X]$.

Since $q\geq 6$, $c\geq 2$, and $r\geq 8$ so $r+1<w$
\begin{align*}
    w=&cr-2c-2\\
    =&(c-2)(r-2)+(r-7)+r+1 \\
    > &r+1 .
\end{align*}
Thus, the coset representative, $\overline{a^3}$ is $b^{r+1}$.

We apply Reidemeister-Schreier to find a basis $\{x_1, \ldots, x_{w},y_1,\ldots,y_{w},z_0,\ldots,z_{w}\}$ where
\[
    x_i=
    \left\{
    \begin{array}{ll}
        b^iab^{-i}a^{-1} & \text{if } 1\leq i \leq w-1 \\
        b^w & \text{if } i=w 
    \end{array}
    \right. ,
\]
\[
    y_i=
    \left\{
    \begin{array}{ll}
        ab^iab^{-i}a^{-2} & \text{if } 1\leq i \leq w-1 \\
        ab^wa^{-1} & \text{if } i=w 
    \end{array}
    \right.
\]
and
\[
    z_i=
    \left\{
    \begin{array}{ll}
        a^2b^iab^{-i-r-1} & \text{if } 0 \leq i <w-r-1 \\ 
        a^2b^iab^{w-i-r-1} & \text{if } w-r-1 \leq i < w \\ 
        a^2b^wa^{-2} & \text{if } i=w
    \end{array}
    \right. .
\]

Using this basis,
\[
    \beta_K\alpha_K=z_0^{-1}
\]
and
\[
    \beta_K=z_{0}^{-1} y_{1}^{-1} x_{2}^{-1} \prod_{i=1}^{c-1} (z_{\delta(i)}^{-1} y_{\delta(i)+1}^{-1} x_{\delta(i)+2}^{-1}) z_{w-2}^{-1} y_{w-1}^{-1} y_{w}
\]
where
\[
    \delta(i)=w-i(r-2) .
\]

Since $w=c(r-2)-2$,
\[
    \delta(i)-1=w-i(r-2)-1=(r-2)(c-i)-3
\]
so if $\delta(i)=1$ then $(r-2)(c-i)=3$.
However, since $r\geq 7$, $r-2$ does not divide 3
so $\delta(i)$ is never 1
so $y_1$ only appears once in $\beta_K$.

Thus, the set $\{\beta_K\alpha_K,\beta_K,x_1,\ldots,x_{w},y_2,\ldots,y_{w},z_1,\ldots,z_{w}\}$ is a free basis of $K[X,X]$.
Since $\{\beta_K\alpha_K,\beta_K\}$ is a free basis of $K$,
$K$ is a free factor of $K[X,X]$.
\end{proof}

%%%%%%%%%%%%%%%%%%%%%%%%%%%%%%%%%%%%%%%%%%%%%%%%%%%%%%%%%%%%%%%%%%%%%%
\section{Chart of Results} \label{seccharts}
%%%%%%%%%%%%%%%%%%%%%%%%%%%%%%%%%%%%%%%%%%%%%%%%%%%%%%%%%%%%%%%%%%%%%%

\begin{table}[t]

\centering

$P(-3,Q,R)$ Pretzel Knots

\includegraphics[scale=1.0]{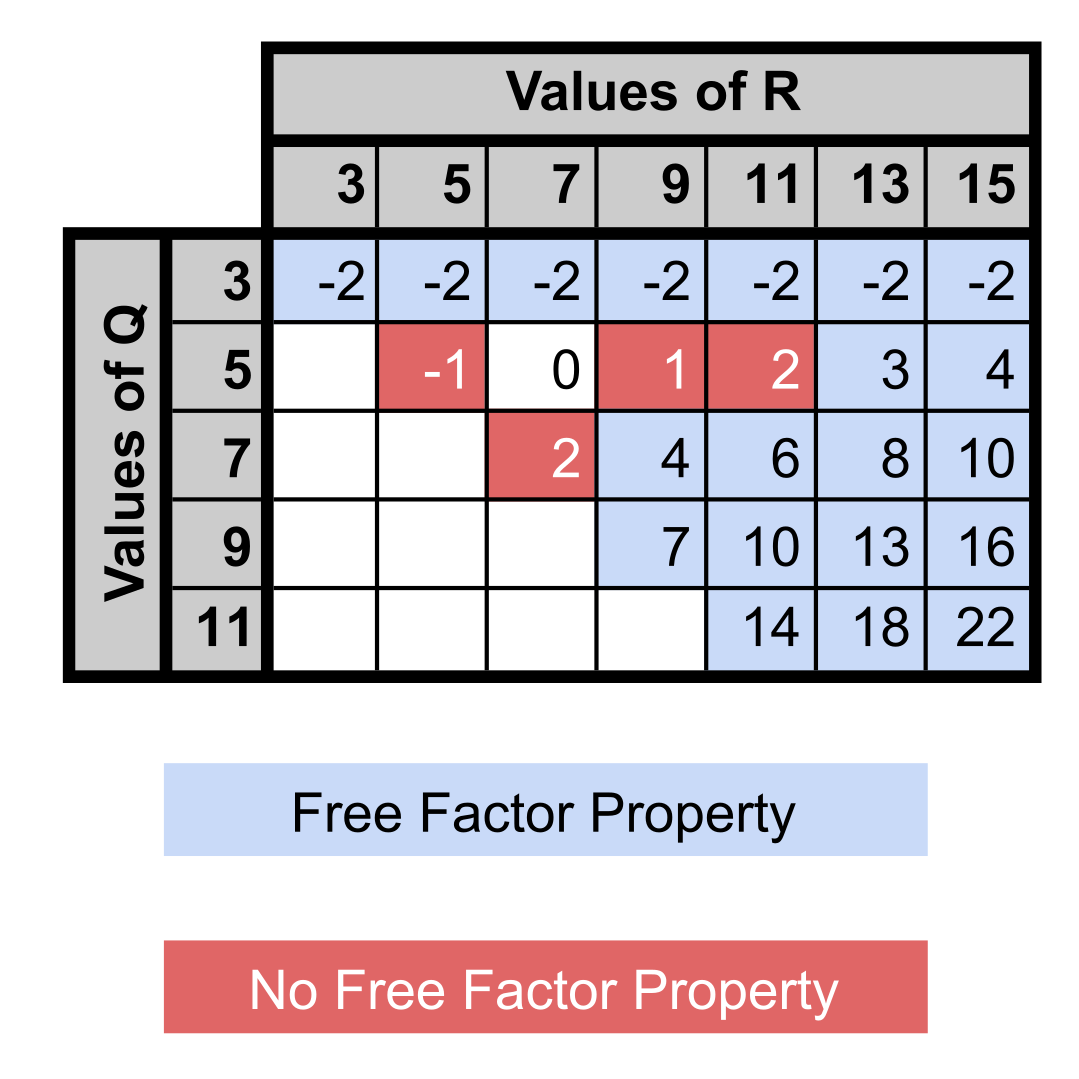}

$P(-5,Q,R)$ Pretzel Knots

\includegraphics[scale=1.0]{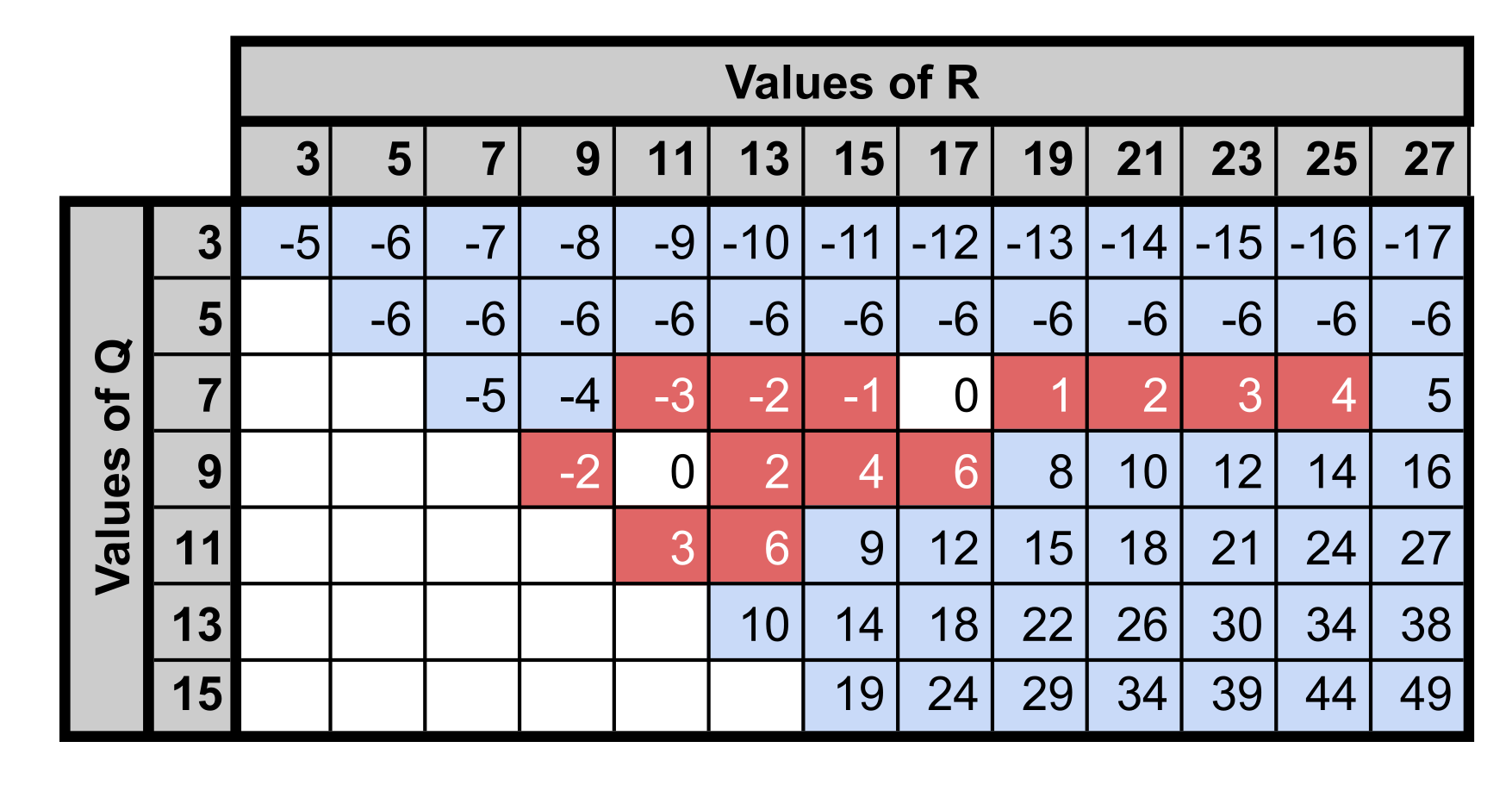}

\caption{The results for some $P(-3,Q,R)$ (top) and $P(-5,Q,R)$ (bottom) pretzel knots where $Q=2q+1$ and $R=2r+1$.
The integers in each cell is the value of $N$.
Each cell is shaded light blue if the knot's standard Seifert surface satisfies the free factor property,
and shaded dark red if the knot's standard Seifert surface does not satisfy the free factor property.}
\label{results_tbl}
\end{table}

Table \ref{results_tbl} summarize the results we've found for the pretzel knots $P(-3,Q,R)$ and $P(-5,Q,R)$ where $Q=2q+1$ and $R=2r+1$.
The shading of the cells, in each chart, indicate whether or not the knot's standard Seifert surface $S$ satisfies the free factor property.
Cells of knots with trivial Alexander polynomial are not shaded.
The integers in each cell is the value of $N=\det(S_+)=\det(S_-)$,
which is also the leading coefficient of the Alexander polynomial.
If a pretzel knot's cell is shaded light blue and $N$ is a prime power, then the knot's group has residually torsion-free nilpotent commutator subgroup.
If in addition, $N<0$, then the knot's group is bi-orderable.

\printbibliography

\end{document}